\documentclass[hidelinks,onefignum,onetabnum]{siamart220329}

\usepackage{xcolor}
\usepackage{amsfonts}
\usepackage{graphicx}
\usepackage{epstopdf}
\usepackage{algorithmic}
\usepackage{enumerate}
\usepackage[sort,compress]{cite} 
\ifpdf
  \DeclareGraphicsExtensions{.eps,.pdf,.png,.jpg}
\else
  \DeclareGraphicsExtensions{.eps}
\fi

\def\R{\mathbb{R}}
\def\N{\mathbb{N}}

\def\K{\mathcal{K}}
\def\Ki{\K_\infty}
\def\KL{\mathcal{KL}}
\def\U{\mathcal{U}}
\def\X{\mathcal{X}}
\def\C{\mathcal{C}}
\def\Ui{\mathsf{U}}
\def\zi{\mathbf{0}}
\def\comp{{\scriptstyle{\,\circ\,}}}
\def\B{\mathcal{B}}

\newsiamremark{remark}{Remark}
\newsiamremark{clm}{Claim}
\newsiamthm{defin}{Definition}

\newtheorem{ex}{Example}[section]
\newtheorem{teo}{Theorem}[section]
\newtheorem{as}{Assumption}
\newtheorem{lema}{Lemma}[section]
\newtheorem{prop}{Proposition}[section]
\newtheorem{rem}{Remark}[section]

\headers{iISS characterization for infinite-dimensional systems}{J.L. Mancilla-Aguilar, J.E. Rojas Ruiz, and H. Haimovich}

\title{Characterization of integral input-to-state stability for nonlinear time-varying systems of infinite dimension\thanks{Submitted to the editors November 29, 2022.
\funding{Work partially supported by Agencia I+D+i grant PICT 2018-01385, Argentina.}}}

\author{Jos\'e L. Mancilla-Aguilar\thanks{Departamento de Matem\'atica, Facultad de Ingenier\'{\i}a, UBA 
  (\email{jmancil@fi.uba.ar}).}
\and Jos\'e E. Rojas-Ruiz\footnotemark\thanks{Instituto Tecnol\'ogico de Buenos Aires, Argentina (\email{jrojas@itba.edu.ar}).}
\and Hernan Haimovich\thanks{Centro Internacional Franco-Argentino de Ciencias de la Informaci\'on y Sistemas (CIFASIS), UNR-CONICET, Rosario, Argentina
  (\email{haimovich@cifasis-conicet.gov.ar}).}}

\usepackage{amsopn}

\ifpdf
\hypersetup{
  pdftitle={iISS Characterization for Infinite Dimensions},
  pdfauthor={J.L. Mancilla-Aguilar, J.E. Rojas Ruiz and H. Haimovich}
}
\fi

\begin{document}

\maketitle

\begin{abstract}
  For large classes of infinite-dimensional time-varying control systems, the equivalence between integral input-to-state stability (iISS) and the combination of global uniform asymptotic stability under zero input (0-GUAS) and uniformly bounded-energy input/bounded state (UBEBS) is established under a reasonable assumption of continuity of the trajectories with respect to the input, at the zero input. By particularizing to specific instances of infinite-dimensional systems, such as time-delay, or semilinear over Banach spaces, sufficient conditions are given in terms of the functions defining the dynamics. In addition, it is also shown that for semilinear systems whose nonlinear term satisfies an affine-in-the-state norm bound, it holds that iISS becomes equivalent to just 0-GUAS, a fact known to hold for bilinear systems. An additional important aspect is that the iISS notion considered is more general than the standard one.
\end{abstract}



\section{Introduction}
Analyses and characterizations of input-to-state stability (ISS) and integral-ISS (iISS) for infinite-dimensional systems, such as time-delay systems, systems modelled by partial differential equations (PDEs) and semilinear systems on Banach spaces, have seen great progress mostly in the last decade 
\cite{Pepe2006, karpep_ejc08, Pepe2013, kanlin_ifacol17, linwan_CDC18, chagok_tac21, karkrs_tac16, karkrs_jco17, karkrs_book18, dasmir_mcss13, 
  dasmir_siamjco13, 
  mirito_siamjco15, 
  mirito_mcrf16, 
  mironc_scl16, 
  Mironchenko2018a, 
  Mironchenko2018, 
  mironc_tac19, 
  mironc_ifacol20, 
  jacmir_siamjco20, 
  mironc_jmci21, jacnab_jco18, 
%
  nabsch_mcss18, 
  schmid_mcss19, 
  hosjac_mcss22}. 
The reader may consult the excellent survey \cite{mirpri_siamrev20} for an 
updated account of results on ISS of infinite dimensional systems. As far as generality is concerned, arguably much greater progress has been made in the analysis and characterization of ISS \cite{mirpri_siamrev20}, as opposed to iISS. For example, for large classes of infinite-dimensional systems not restricted to semilinear systems over Banach spaces, \cite{Mironchenko2018} characterizes ISS in terms of simpler properties and \cite{mironc_tac19} does so for input-to-state practical stability. As mentioned in \cite{Mironchenko2018}, characterizations of ISS in terms of other simpler stability properties are advantageous in simplifying proofs and in analysing different classes of systems. 

As regards characterizations of iISS for time-invariant infinite dimensional systems, \cite{linwan_CDC18, chagok_tac21} characterize iISS for time-delay systems in terms of Lyapunov-Krasovskii functionals. For linear (infinite-dimensional) evolution equations on Banach spaces with bounded input operators, it is known that ISS, iISS and uniform global asymptotic stability under zero input become equivalent \cite{mirito_mcrf16}, analogously to finite-dimensional linear systems. More generally, for bilinear infinite-dimensional systems on Banach spaces \cite{mirito_mcrf16} establishes the equivalence between iISS and uniform global asymptotic stability under zero input and establishes the existence of iISS-Lyapunov functions in the case of Hilbert spaces, under additional assumptions. As for linear infinite-dimensional systems with unbounded input operators, \cite{jacnab_jco18} characterizes iISS in terms of the exponential stability of the semigroup and an admissibility condition on the inputs. In \cite{hosjac_mcss22}, the iISS of bilinear systems with unbounded operators is characterizated in terms of the iISS of certain associated linear systems.
To the best of the authors' knowledge, other more general characterizations have not yet been developed, nor characterizations valid for time-varying systems. One problem is that many of the characterizations developed in \cite{Angeli2000} for time-invariant finite-dimensional systems cease to hold already when the finite-dimensional system is time-varying or when the setting is such that existence of an iISS-Lyapunov function is not guaranteed \cite{haiman_tac18}. One characterization that remains valid in these cases is the superposition-type one 
\begin{align}
  \label{eq:iISS-superp}
  \text{``iISS $\Leftrightarrow$ 0-GUAS $\wedge$ UBEBS''},
\end{align}
stating that iISS is equivalent to the combination of global uniform asymptotic stability under zero input (0-GUAS) and uniform bounded-energy input/bounded state (UBEBS) \cite{haiman_tac18, haiman_auto20}, as originally stated for time-invariant systems in \cite{Angeli2000a}.

In this context, the contribution of the current paper is to show that the characterization of iISS as the combination of 0-GUAS and UBEBS remains valid for broad classes of time-varying infinite-dimensional systems, provided a reasonable condition of continuity with respect to the input, uniformly with respect to initial time, is satisfied by the system trajectories, at the zero input. This characterization is established with a focus on minimizing assumptions on the input so that, in addition to the standard iISS notion involving an integral of a function of the input, the characterization also holds for more general notions of iISS and UBEBS. By particularizing to specific classes of systems, such as semilinear over Banach spaces or to retarded ordinary differential equations, simpler sufficient conditions to ensure the required continuity with respect to the input are also given.

The organization of the paper is as follows. Section~\ref{sec:preliminaries} gives the definitions of time-varying system with inputs and the required stability properties, and poses the specific problem addressed. In Section~\ref{sec:main-result:-char}, the required assumption of continuity with respect to the input is given and the equivalence~(\ref{eq:iISS-superp}) is established. Sections~\ref{sec:time-delay} and~\ref{sec:semilinear-systems} provide simpler sufficient conditions to ensure the required continuity in the case of time-delay systems and semilinear systems on Banach spaces, respectively. Section~\ref{sec:semilinear-systems} also contains the particular case where iISS becomes equivalent to just 0-GUAS. Conclusions and final remarks are given in Section~\ref{sec:conclusions}. Most proofs are given in the Appendices.

\subsection*{Notation}  
$\N$, $\R$, $\R_{>0}$ and $\R_{\ge 0}$ denote the natural numbers, reals, positive, and nonnegative reals, respectively. We write $\alpha\in\K$ if $\alpha:\R_{\ge 0} \to \R_{\ge 0}$ is continuous, strictly increasing and $\alpha(0)=0$, and $\alpha\in\Ki$ if, in addition, $\alpha$ is unbounded. We write $\beta\in\KL$ if $\beta:\R_{\ge 0}\times \R_{\ge 0}\to \R_{\ge 0}$, {\color{black} $\beta(\cdot,t)\in \K$} for any $t\ge 0$ and, for any fixed $r\ge 0$, $\beta(r,t)$ monotonically decreases to zero as $t\to \infty$. The single bars $|\cdot|$ denote the Euclidean norm in $\R^n$. $B_r$ denotes a closed ball of radius $r\ge 0$ centred at $0$.

\section{Preliminaries}
\label{sec:preliminaries}
\subsection{Time-varying systems with inputs}
In order to make our results applicable to different classes of time-varying systems with inputs, for example those described by ordinary differential equations (ODE) with or without impulse effects, retarded differential equations (RDE), semilinear differential equations (SDE) and switched systems, among others, we consider the following general definition of time-varying system with inputs.
\begin{defin} \label{def:system}
Consider the triple $\Sigma=\left(\X ,\U ,\phi \right)$ consisting of
\begin{enumerate}[1.-]
\item A normed vector space $\left (\mathcal{X},\|\cdot\|_{\X}\right )$, which we call the state space.
\item A set of admissible inputs $\U=\{u:\R_{\ge 0}\to \Ui\}$, where $\Ui$ is a vector space of input values, which satisfies:
\begin{enumerate}
\item The zero input belongs to $\U$, i.e. $\zi\in\U$ with $\zi:\R_{\ge 0}\to \Ui$ such that $\zi(t)\equiv  0$.
\item If $u,v\in \U$ and $t>0$, then the concatenation $u\sharp_t  v\in \U$, where
\begin{align*}
u\sharp_t  v(\tau)=\begin{cases} u(\tau)\quad \tau \le t, \\
v(\tau)\quad \tau > t.
\end{cases}
\end{align*}
\end{enumerate}
\item A transition map $\phi:D_{\phi}\rightarrow \X$, with $D_{\phi}\subset \{(t,s,x,u):t\ge s\ge 0, x\in \X, u\in \U\}$, such that for all $s\ge 0$, $x\in \X$ and $u\in \U$,
$\{t\in \R_{\ge 0}:(t,s,x,u)\in D_{\phi}\}=[s,t_{(s,x,u)})$ with $s<t_{(s,x,u)}\le \infty$.
\end{enumerate}
We say that $\Sigma$ is a system with inputs if the following properties hold:
\begin{enumerate}
\item[$(\Sigma1)$]\underline{Identity}: $\phi(t,t,x,u)=x$ for all $t\ge 0$, $x\in \X$ and $u \in \U$.
\item[$(\Sigma2)$] \underline{Causality}: for all $(t,s,x,u) \in D_{\phi}$ with $t>s$, if $v\in \U$ satisfies $v(\tau)=u(\tau)$ for all $\tau \in (s,t]$, then $(t,s,x,v) \in D_{\phi}$ and $\phi(t,s,x,v)=\phi(t,s,x,u)$. 
\item[$(\Sigma3)$]\underline{Semigroup}: for all $(t,s,x,u) \in D_{\phi}$ with $s<t$, if $s<\tau<t$\\ then $\phi(t, \tau,\phi(\tau,s,x,u),u) = \phi(t,s,x,u)$.
\end{enumerate}
\end{defin}

This definition of system involves existence and uniqueness of solutions and is an extension of that in \cite{Mironchenko2018, mirpri_siamrev20} to encompass various classes of time-varying systems. One difference here is that the function $\phi(\cdot,s,x, u)$ is not assumed continuous for every fixed $(s,x,u)$ (cf. property $\Sigma_3$ in \cite{mirpri_siamrev20}). This allows for the occurrence of jumps in the state trajectory. For other definitions of systems with inputs the reader may consult \cite{Sontag1998a, Iasson2011}. 

Given $t_0\ge 0$, $x_0\in \X$ and $u\in \U$, the function $x(t)=\phi(t,t_0,x_0,u)$, $t\in [t_0,t_{(t_0,x_0,u)})$, will be referred to as the trajectory of $\Sigma$ corresponding to the initial time $t_0$, initial state $x_0$ and input $u$. We say that $\Sigma$ is forward complete if for all $t_0\ge 0$, $x_0\in \X$ and $u\in \U$, $t_{(t_0,x_0,u)}=\infty$, i.e., if every trajectory is defined for all times $t$ greater than the initial time.

Given an interval $J\subset \R_{\ge 0}$ and $u \in \U$ we define 
$u_J:\R_{\ge 0}\to \Ui$ via $u_J(\tau)=u(\tau)$ if $\tau \in J$ and $u_J(\tau)=0$ otherwise. Since $u_{(s,t]}=\zi \sharp_s u \sharp_t \zi$ and $u_{(s,\infty)}=\zi \sharp_s u $, both $u_{(s,t]}$ and $u_{(s,\infty)}$ belong to $\U$. Due to causality, if $(t,s,x,u)\in D_{\phi}$ and $t>s$ then $\phi(\tau,s,x,u_{(s,t]})=\phi(\tau,s,x,u_{(s,\infty)})=\phi(\tau,s,x,u)$ for all $\tau \in (s,t]$.


\subsection{Stability definitions}
The stability properties considered next are straightforward extensions of those defined for specific classes of systems with inputs. The set of admissible inputs $\U$ is assumed to be endowed with a nonnegative admissible functional, defined as follows.
\begin{defin} \label{def:admissible-f}
  The functional $\|\cdot\|_{\U}:\U\to \R_{\ge 0}\cup \{\infty\}$ is said to be \emph{admissible} if it satisfies the following conditions:
\begin{enumerate}[a)]
\item \label{item:zero-input prop} $\|\zi\|_{\U}=0$;
\item \label{item: monotonicity} for all $0\le s<t <\infty$ and all $u\in \U$, $\|u_{(s,t]}\|_{\U}<\infty$ and $\|u_{(s,t]}\|_{\U}\le \|u_{(s,\infty)}\|_{\U}\le \|u\|_{\U}$.
\end{enumerate}
\end{defin}
An admissible functional $\|\cdot\|_{\U}$ is not required to be a norm on $\U$, and is not even necessarily finite for every input $u$. The set of inputs $u\in \U$ for which $\|u\|_{\U}<\infty$ will be denoted by $\U_F$. {\color{black} The functionals defined below are examples of admissible functionals.
\begin{defin} \label {def:usual-functionals} Let $\Ui$ be a normed space with norm $\|\cdot\|_{\Ui}$ and let $L^{\infty}_{\mathrm{loc}}(\R_{\ge 0},\Ui)=\{u:\R_{\ge 0}\to \Ui:\|u(\cdot)\|_{\Ui}\in L^{\infty}_{\mathrm{loc}}(\R_{\ge 0})\}$, where $L^{\infty}_{\mathrm{loc}}(\R_{\ge 0})$ is the set of locally essentially bounded Lebesgue measurable functions $h:\R_{\ge 0}\to \R$. For a set of admissible inputs $\U\subset L^{\infty}_{\mathrm{loc}}(\R_{\ge 0},\Ui)$ we define the following functionals:
\begin{enumerate}[a)]
\item $\|u\|_{\infty}:=\mathrm{ess\:sup}_{t\ge 0}\|u(t)\|_{\Ui}$.
\item Given a positive, strictly increasing and unbounded sequence $\lambda=\{\tau_k\}$, $\|u\|_{\infty,\lambda}:=\|u\|_{\infty}+\sup_{k}\|u(\tau_k)\|_{\Ui}$.
\item Given $\kappa\in \K$, $\|u\|_{\kappa}:=\int_0^{\infty}\kappa\left(\|u(s)\|_{\Ui}\right)\:ds$.
\item For $\kappa$ and $\lambda$ as above, $\|u\|_{\kappa,\lambda}:=\|u\|_{\kappa}+\sum_{k}\kappa(\|u(\tau_k)\|_{\Ui})$.
\item For $\kappa$ as above and $T>0$, $\|u\|_{\kappa,T}:=\sup_{t\ge 0}\|u_{(t,t+T]}\|_{\kappa}$.
\end{enumerate}
\end{defin}
}
The following stability properties are extensions to time-varying systems of some of those in \cite{Mironchenko2018}. Since $\|u\|_{\U}=\infty$ may be true for some inputs $u\in \U$, we adopt the convention $\rho(\infty)=\infty$ for any function $\rho \in \Ki$. 
\begin{defin}\label{def:stability-def}
Let $\Sigma$ be a system with inputs and let $\|\cdot\|_{\U}$ be an admissible functional. Then
\begin{enumerate}[a)]
\item \label{item:0-GUAS}$\Sigma$ is zero-input globally uniformly asymptotically stable (0-GUAS) if there exists $\beta \in \KL$ such that for all $t_0\ge 0$ and $x_0\in \X$, the trajectory $x(t)=\phi(t,t_0,x_0,\zi)$ is defined for all $t\ge t_0$ and satisfies
\begin{align} \label{eq:0-guas}
\|x(t)\|_{\X}\le \beta(\|x_0\|_{\X} ,t-t_0) \quad \forall t\ge t_0.
\end{align}
\item \label{item:ISS}$\Sigma$ is input-to-state stable (ISS) with respect to the admissible functional $\|\cdot\|_{\U}$, abbreviated $\|\cdot\|_{\U}$-ISS, if $\Sigma$ is forward complete and there exist $\rho\in \Ki$ and $\beta \in \KL$ such that for all $t_0\ge 0$, $x_0\in \X$ and $u\in \U$, the corresponding trajectory $x(\cdot)$ of $\Sigma$ satisfies
\begin{align} \label{eq:iss}
\|x(t)\|_{\X}\le \beta(\|x_0\|_{\X},t-t_0)+\rho(\|u\|_{\U}) \quad \forall t\ge t_0.
\end{align}
\item \label{item:UGB}$\Sigma$ is uniformly globally bounded (UGB) with respect to the admissible functional $\|\cdot\|_{\U}$, abbreviated $\|\cdot\|_{\U}$-UGB, if $\Sigma$ is forward complete and there exist $\alpha, \rho\in \Ki$ and $c\ge 0$ such that for all $t\ge t_0\ge 0$, $x_0\in \X$ and $u\in \U$, the corresponding trajectory $x(\cdot)$ of $\Sigma$ satisfies
\begin{align} \label{eq:ugb}
\|x(t)\|_{\X}\le c+\alpha(\|x_0\|_{\X})+\rho(\|u\|_{\U})\quad \forall t\ge t_0.
\end{align}
\item \label{item:UGS}$\Sigma$ is uniformly globally stable (UGS) with respect to the admissible functional $\|\cdot\|_{\U}$, abbreviated $\|\cdot\|_{\U}$-UGS, if it is $\|\cdot\|_{\U}$-UGB and (\ref{eq:ugb}) holds with $c=0$.
\end{enumerate}  
\end{defin}
The word ``uniformly'' in Defintion~\ref{def:stability-def}~\ref{item:0-GUAS}),  \ref{item:UGB})~and~\ref{item:UGS}) involves uniformity both with respect to the state and with respect to initial time. For conciseness, we avoid the use of a double `U' and use `0-GUAS' instead of the `0-UGAS' used to denote uniformity with respect to the state in, e.g. \cite{mironc_scl16}. 
Whenever the admissible functional $\|\cdot\|_{\U}$ is clear from the context, we may remove the prefix $\|\cdot\|_{\U}$ and simply refer to the ISS, UGB or UGS properties.

Note the following:
\begin{enumerate}[i)]
\item \label{item:truncateu}Due to causality ($\Sigma 2$) and Definition~\ref{def:admissible-f}\ref{item: monotonicity}), replacing $\|u\|_{\U}$ by $\|u_{(t_0,t]}\|_{\U}$ or $\|u_{(t_0,\infty)}\|_{\U}$ in (\ref{eq:iss}) and (\ref{eq:ugb}), equivalent definitions of ISS and UGB (or UGS), respectively, are obtained.
\item Since (\ref{eq:iss}) and (\ref{eq:ugb}) are trivially satisfied when $\|u\|_{\U}=\infty$, no loss of generality is incurred if only inputs belonging to $\U_F$ are considered in the definitions of ISS and UGB.
\item When the system $\Sigma$ satisfies the boundedness-implies-continuation (BIC) property, {\em i.e.} when $\phi(\cdot,s,x,u)$ being bounded on $[s,t_{(s,x,u)})$ implies that $t_{(s,x,u)}=\infty$, then the forward completeness requirement can be removed from Definition~\ref{def:stability-def}. This happens because, since from item~\ref{item:truncateu}) above $\|u\|_{\U}$ can be replaced by $\|u_{(t_0,t]}\|_{\U}$ and $\|u_{(t_0,t]}\|_{\U} < \infty$ from Definition~\ref{def:admissible-f}\ref{item: monotonicity}), then the satisfaction of (\ref{eq:iss}) or (\ref{eq:ugb}) for all $t \in [s,t_{(s,x,u)})$ and all $s\ge 0$, $x\in \X$ and $u\in \U$ would imply that $t_{(s,x,u)}=\infty$ and therefore that $\Sigma$ is forward complete.
\end{enumerate}

Some standard stability properties defined for specific classes of systems, such as those modelled by ODEs with or without impulse effects, RDEs or PDEs, are recovered by choosing the admissible functional $\|\cdot\|_{\U}$ in a suitable manner. 
For example, for systems without impulse effects, $\|\cdot\|_{\infty}$-ISS is the standard ISS property and $\|\cdot\|_{\infty}$-UGS is the uniform bounded-input bounded-state property \cite{bacmaz_JCO00}. Moreover, for $\kappa \in \Ki$, then $\|\cdot\|_{\kappa}$-ISS becomes iISS, and $\|\cdot\|_{\kappa}$-UGB and $\|\cdot\|_{\kappa}$-UGS become uniformly bounded-energy input/bounded-state (UBEBS) and UBEBS with constant $c=0$, respectively (see \cite{Angeli2000a, haiman_tac18, Pepe2006, chagok_tac21}). In these cases, $\kappa$ is referred to as the iISS- or UBEBS-gain according to the considered property. Also, $\|\cdot\|_{\kappa,T}$-ISS is an extension of the p-ISS property considered in \cite{manhai_tac17}. 
In the case of systems with impulse effects, where the state jumps at a fixed sequence $\lambda$ of impulse-time instants, $\|\cdot\|_{\infty,\lambda}$-ISS, $\|\cdot\|_{\kappa,\lambda}$-ISS, $\|\cdot\|_{\kappa,\lambda}$-UGB and $\|\cdot\|_{\kappa,\lambda}$-UGS become, respectively, the usual ISS, iISS, UBEBS and UBEBS with constant $c=0$ properties and in the case of the iISS and UBEBS properties $\kappa$ is also referred to as the iISS- and UBEBS-gain \cite{heslib_auto08, manhai_tac20, haiman_rpic19, haiman_auto20}. 

A common feature of $\|\cdot\|_{\kappa}$ and $\|\cdot\|_{\kappa,\lambda}$ is that both functionals satisfy the following condition (actually with equality):
\begin{itemize}
\item[(E)] For every $u\in \U$ and $0\le t_1< t_2 < t_3$,  $\|u_{(t_1,t_3]}\|_{\U}\ge \|u_{(t_1,t_2]}\|_{\U}+\|u_{(t_2,t_3]}\|_{\U}$.
\end{itemize}

\begin{defin}
  \label{def:iISS}
  Let $\Sigma$ be a system with inputs and let $\|\cdot\|_{\U}$ be an admissible functional that satisfies condition (E).
  \begin{itemize}
  \item If $\Sigma$ is $\|\cdot\|_{\U}$-ISS, then we say that $\Sigma$ is $\|\cdot\|_{\U}$-iISS
  \item If $\Sigma$ is $\|\cdot\|_{\U}$-UGB, then we say that $\Sigma$ is $\|\cdot\|_{\U}$-UBEBS.
    \item If $\Sigma$ is $\|\cdot\|_{\U}$-UGS, then we say that $\Sigma$ is $\|\cdot\|_{\U}$-UBEBS with constant $c=0$, or just $\|\cdot\|_{\U}$-UBEBS0.
  \end{itemize}
  We remove the prefix $\|\cdot\|_{\U}$ when this is clear from the context. In addition, when $\|\cdot\|_{\U}=\|\cdot\|_{\kappa}$ for some $\kappa\in \Ki$, we refer to $\kappa$ as the iISS, UBEBS or UBEBS0 gain.
\end{defin}

      
\subsection{Problem statement}
It is clear from the very definitions that $\|\cdot\|_{\U}$-iISS implies $0$-GUAS and $\|\cdot\|_{\U}$-UBEBS. The aim of the current paper is to investigate the converse implication. 

Conditions that ensure that 0-GUAS and UBEBS imply iISS are known for systems generated by ODEs with or without impulse effects \cite{Angeli2000a, haiman_tac18, haiman_auto20} and for time-invariant time-delay systems \cite{chagok_tac21}. These conditions involve assumptions on the functions appearing in the equations that define the systems. More specifically, such functions must have some type of regularity and satisfy specific bounds. These conditions suggest that, for the kind of general system considered here, the transition map $\phi$ is required to have some specific regularity. 
  
The following example gives some insight into the type of regularity which may be required.
\begin{ex} \label{ex:regularity}
Consider the system $\Sigma=(\X,\U,\phi)$ with $(\X,\|\cdot\|_{\X})=(\Ui,\|\cdot\|_{\Ui})=(\R,|\cdot|)$, $\U$ the set of piecewise constant functions $u:\R_{\ge 0}\to \R$ and $\phi:D_{\phi}\to \R$, with $D_{\phi}=\{(t,s) : t\ge s\ge 0\}\times \R \times \U$, defined as follows. 
Pick any smooth function $g:\R \to \R$ such that $0\le g(r)\le 1$ for all $r\in \R$, $g(r)=1$ if $|r|\le 1$ and $g(r)=0$ if $|r|\ge 2$. For a given $(t_0,x_0,u)\in \R_{\ge 0}\times \R \times \U$, let $x(\cdot)$ be the unique solution of the scalar initial value problem
\begin{align} \label{eq:exsys}
\dot x=-x+g(x)v(t),\quad x(t_0)=x_0,
\end{align}
where $v:\R_{\ge 0}\to \R$ is the piecewise constant function defined by $v(t)=1/u(t)$ if $u(t)\neq 0$ and $v(t)=0$ if $u(t)=0$. Note that $x(\cdot)$ is defined for all $t\ge t_0$. Then we define $\phi(t,t_0,x_0,u)=x(t)$ for all $t\ge t_0$. 
It is a simple exercise to show that the triple $(\X,\U,\phi)$ is a system with inputs according to Definition \ref{def:system} and that it is forward complete. 

For $u=\zi$, we have that $v=\zi$, and then $\Sigma$ is $0$-GUAS since the trajectories corresponding to $u$ satisfy the equation $\dot{x}=-x$. From the fact that $g(r)=0$ for all $|r|\ge 2$, it follows that any solution of (\ref{eq:exsys}) satisfies $|x(t)|\le 2+|x_0|$ for all $t\ge t_0$ and therefore $\Sigma$ is $\|\cdot\|_{\kappa}$-UBEBS for any $\kappa \in \Ki$.

Next, we will prove that $\Sigma$ is not $\|\cdot\|_{\kappa}$-iISS for any $\kappa \in \Ki$. Suppose on the contrary that $\Sigma$ is $\|\cdot\|_{\kappa}$-iISS  for some $\kappa \in \K$. Then there exist $\beta \in \KL$ and $\rho \in \Ki$ so that (\ref{eq:iss}) holds, with $\|u\|_{\kappa}$ in place of $\|u\|_{\U}$. 

We claim that for every $\delta>0$ there exists an input $u$ such that $\|u\|_{\kappa}<\delta$ and $|\phi(t,0,0,u)|>\frac{1}{2}$ for some $t>0$.

Let $\mu>0$ be such that $\mu<1-e^{-1}$ and $\kappa(\mu)<\delta$. Define $u(t)=\mu$ if $t\in [0,1]$ and $u(t)=0$ for $t>1$. Then $\|u\|_{\kappa}=\kappa(\mu)<\delta$. Let $x(t)=\phi(t,0,0,u)$ and suppose that $|x(t)|\le \frac{1}{2}$ for all $t\in [0,1]$. From the definition of $\phi$ it follows that $\dot x(t)=-x(t)+\frac{1}{\mu}$ for all $t\in [0,1]$ and that $x(0)=0$. Therefore $x(t)=\int_0^t \frac{e^{-(t-s)}}{\mu}ds=\frac{1-e^{-t}}{\mu}$. In consequence $x(1)=\frac{1-e^{-1}}{\mu}>1$ which is a contradiction. So, there must exist $t\in [0,1]$ such that $|x(t)|>\frac{1}{2}$. This proves the claim.

From the claim it easily follows that $\Sigma$ cannot be $\|\cdot\|_{\kappa}$-iISS, since taking $\delta>0$ such that $\rho(\delta)<\frac{1}{2}$ and $u$ and $t$ as in the claim, then~(\ref{eq:iss}) implies that $\frac{1}{2}<|\phi(t,0,0,u)|\le \rho(\|u\|_{\kappa})<\frac{1}{2}$, which is absurd.
\end{ex}

Note that the transition map $\phi(t,t_0,x_0,u)$ in the preceding example is continuous in $(t,t_0,x_0)$ for any fixed $u\in \U$ but, due to the claim above, it is not continuous with respect to the input $u$ when $u$ is near the zero input $\zi$ (i.e. when $\|u\|_{\kappa}$ is small). This suggests that for the problem to have a solution some continuity condition on the map $\phi$ with respect to small inputs $u$ may be required.

The more specific problem addressed is hence the following:

{\em Find conditions on the transition map $\phi$ that ensure that the $0$-GUAS and $\|\cdot\|_{\U}$-UBEBS  of $\Sigma$ imply the $\|\cdot\|_{\U}$-iISS of $\Sigma$.}

\section{Main result: a characterization of iISS}
\label{sec:main-result:-char}
By solving the previous problem, the characterization of iISS as the superposition~(\ref{eq:iISS-superp}) will be extended to general classes of infinite-dimensional systems. The following condition on the transition map will be required.
\begin{as}
  \label{as:transition-map}
  The transition map $\phi$ of the system $\Sigma$ satisfies the following:\\
  For every $r>0$, $\varepsilon>0$ and $T>0$ there exists $\delta=\delta(r,\varepsilon,T)>0$ such that for every $t_0\ge 0$, $x_0 \in \X$ and $u\in \U$ with $\|u\|_{\U}\le \delta$, if for some $t^*\in [t_0, t_0+T]$ it happens that $\|x(t)\|_{\X}\le r$ and $\|z(t)\|_{\X}\le r$ for all $t\in [t_0,t^*]$, where $x(t)=\phi(t,t_0,x_0,u)$ and $z(t)=\phi(t,t_0,x_0,\zi)$, then 
\begin{align} \|z(t)-x(t)\|_{\X} \le \varepsilon\quad \forall t\in [t_0,t^*].
\end{align}
\end{as}
Assumption~\ref{as:transition-map} means that the solution $x$ corresponding to an input can be made arbitrarily close to the zero-input solution $z$ by reducing the input, as measured by the admissible functional, whenever both solutions remain bounded by $r$ over some time interval of prespecified maximum length $T$. This should happen uniformly over the initial time.

{\color{black} Assumption \ref{as:transition-map} is satisfied by some general classes of time-varying systems such as those described by ODEs and RDEs (as shown in Section~\ref{sec:time-delay}, where the ODE case is covered by RDEs with maximum delay $0$) and SDEs (Section~\ref{sec:semilinear-general}), assuming that the admissible functional is of the type $\|\cdot\|_{\kappa}$ and that the functions defining the dynamics satisfy some suitable boundedness and regularity conditions. 
  Assumption \ref{as:transition-map} can also be proved to hold for systems described by ODEs with impulse effects when the sequence of impulse times is fixed and the admissible functional is of the type $\|\cdot\|_{\kappa,\lambda}$. This can be done using results and techniques in \cite{haiman_auto20} whenever the sequence of impulse times satisfies the uniform incremental boundedness (UIB) property defined in that paper.}

{\color{black} When the system is forward complete and has the UBRS property defined next, Assumption \ref{as:transition-map} can be formulated equivalently in much simpler form, as the following lemma shows. The proof is given in Appendix~\ref{app:proof-lem-ass1}. Comments on the UBRS property are given later, after Theorem~\ref{thm:ISS-e-d}.
  \begin{lema} \label{lem:ass1} Let $\|\cdot\|_{\U}$ be an admissible functional. Suppose that $\Sigma$ is forward complete and has the 
    \begin{enumerate}[C1)]
    \item \label{item:ubrs} uniformly bounded reachability sets (UBRS) property: For every $T>0$, $r>0$ and $s>0$ there exists $C\ge 0$ such that for all $t_0\ge 0$, $x_0\in \X$ with $\|x_0\|_{\X} \le r$ and $u\in\mathcal{U}$ with $\|u\|_{\U}\leq s$, then $\|\phi(t,t_{0},x_0,u)\|_{\mathcal{X}}\leq C$ for all $t\in[t_{0},t_{0}+T]$.
    \end{enumerate}
    Then, the following are equivalent:
\begin{enumerate}[a)]
\item \label{item:I} $\Sigma$ satisfies Assumption \ref{as:transition-map}.
\item \label{item:II} There exists $\Gamma:\R^3_{\ge 0}\to \R_{\ge 0}$ such that $\Gamma$ is continuous and nondecreasing in each of its first two arguments, of class $\Ki$ in the third argument and the following holds: 
if $x(t)=\phi(t,t_0,x_0,u)$ and $z(t)=\phi(t,t_0,x_0,\zi)$ for some $t_0\ge 0$, $x_0\in \X$ and $u\in \U$, then 
\begin{align} \label{eq:ass1} \|z(t)-x(t)\|_{\X} \le \Gamma(t-t_0,\|x_0\|_{\X},\|u_{(t_0,t]}\|_{\U})\quad \forall t\ge t_0.
\end{align}
\end{enumerate}
\end{lema}
}
The following theorem is our main result.
\begin{teo} \label{thm:main}
Let $\Sigma$ be a forward complete system endowed with an admissible functional $\|\cdot\|_{\U}$ satisfying condition (E). Let Assumption \ref{as:transition-map} hold. Then the following are equivalent:
\begin{enumerate}[(a)]
\item \label{item:a} $\Sigma$ is $0$-GUAS and UBEBS.
\item \label{item:b} $\Sigma$ is iISS.
\end{enumerate}
\end{teo}
The proof of Theorem \ref{thm:main} employs the $\varepsilon$-$\delta$ characterization of the ISS property provided by Theorem \ref{thm:ISS-e-d} and Lemma \ref{lem:o-guas+UGR:UGS}, whose proofs are given in the Appendix. This $\varepsilon$-$\delta$ characterization applies to the general ISS property in Definition~\ref{def:stability-def}\ref{item:ISS}) where the input functional should be admissible but is not required to satisfy condition (E). In what follows, $B_r$ denotes the closed ball of radius $r\ge 0$ centred at $0$ in $\X$, namely $B_r = \{x \in \X : \|x\|_{\X} \le r\}$. 
\begin{teo}\label{thm:ISS-e-d}
Let $\Sigma$ be a forward complete system and let $\|\cdot\|_{\U}$ be an admissible functional. Then $\Sigma$ is $\|\cdot\|_{\U}$-ISS if and only if the following properties hold:
\begin{enumerate}[C1)] {\color{black} 
\item Uniformly bounded rechability sets (UBRS), as defined in Lemma~\ref{lem:ass1}.
\item Uniform continuity at the equilibrium point (UCEP): \label{item:smallx0u}For all $h>0$ and $\varepsilon >0$ there exists $\delta>0$ such that for every $t_0\ge 0$, $x_0\in B_{\delta}$ and $u\in\mathcal{U}$ with $\|u\|_{\U}\le \delta$, 
$\|\phi(t,t_{0},x_0,u)\|_{\mathcal{X}}\leq \varepsilon$ for all $t\in [t_0,t_0+h]$.
\item \label{item:attract} Uniform (w.r.t.~initial time) uniform (w.r.t.~initial state) asymptotic gain (UUAG):} There exists $\nu\in\K$ such that for all $r\geq\varepsilon>0$ there is a positive $T=T(r,\varepsilon)$ so that the following holds: for every $t_0\ge 0$, $x_0\in B_r$ and $u\in\mathcal{U}$ we have that
$\|\phi(t,t_{0},x_0,u)\|_{\mathcal{X}}\leq \varepsilon+\nu\left(\|u\|_{\U}\right)$ for all $t \ge t_{0}+T$.
\end{enumerate}
\end{teo}
{\color{black} The UBRS, UCEP and UUAG properties are generalizations to time-varying systems of BRS, CEP and UAG as defined in \cite{Mironchenko2018} for time-invariant infinite-dimensional systems. These properties are generalized so that they are uniform with respect to initial time. When particularized to time-invariant systems, UBRS, UCEP and UUAG are still more general than BRS, CEP and UAG of \cite{Mironchenko2018} because the input functional $\|\cdot\|_{\U}$ is not required to be a norm. 
Theorem~\ref{thm:ISS-e-d} generalizes the equivalence between items i) and ii) in \cite[Thm. 5]{Mironchenko2018}, namely 
\begin{align*}
  \text{ISS $\Leftrightarrow$ UAG $\wedge$ CEP $\wedge$ BRS},
\end{align*}
on the one hand by allowing time-varying systems and on the other by considering a more general definition of ISS that incorporates iISS within a unifying framework.

Theorem \ref{thm:ISS-e-d} is also a generalization of the $\varepsilon$-$\delta$ characterization of ISS in Lemma 2.7 of \cite{sonwan_scl95}. The property UBRS is not explicitly stated in Lemma~2.7 of \cite{sonwan_scl95} because it is automatically satisfied for time-invariant finite-dimensional systems defined by $\dot{x}=f(x,u)$ with $f$ locally Lipschitz in $(x,u)$. }

The proof of Theorem \ref{thm:ISS-e-d} is inspired in the proofs of Lemma 2.7 of \cite{sonwan_scl95} and of Theorem 5 in \cite{Mironchenko2018} and is provided for the sake of completeness in Appendix~\ref{app:proof-thm:ISS-e-d}. 

The proof of our main result, Theorem~\ref{thm:main}, requires the following two lemmas. The first one shows that under the continuity with respect to the input provided by Assumption~\ref{as:transition-map}, then 0-GUAS $\wedge$ UGB $\Rightarrow$ UGS. Note that when the input functional satisfies condition (E), then the latter reads as 0-GUAS $\wedge$ UBEBS $\Rightarrow$ UBEBS0 (Definition~\ref{def:iISS}). The second lemma gives a specific bound for the trajectories of a 0-GUAS and forward complete system that satisfies Assumption~\ref{as:transition-map}.
\begin{lema}
  \label{lem:o-guas+UGR:UGS}
  Let $\Sigma$ be a system and let $\|\cdot\|_{\U}$ be an admissible functional. Let Assumption~\ref{as:transition-map} hold. If $\Sigma$ is 0-GUAS and $\|\cdot\|_{\U}$-UGB then it is $\|\cdot\|_{\U}$-UGS.
\end{lema}
The proof of Lemma~\ref{lem:o-guas+UGR:UGS} is given in Appendix~\ref{app:proof-lem:o-guas+UGR:UGS}.
\begin{lema}
  \label{lem:main}
  Let $\Sigma$ be a forward complete 0-GUAS system endowed with an admissible functional $\|\cdot\|_{\U}$. Let Assumption~\ref{as:transition-map} hold. Then, for every $r>0$, $\eta>0$ and $T>0$, there exists $\gamma=\gamma(r,\eta,T)>0$ such that if $\|\phi(t,t_{0},x,u)\|_{\mathcal{X}}\leq r$ for all $t\in[t_{0},t_{0}+T]$ and $\|u\|_{\U}\leq \gamma$ then
\begin{align} \label{eq:beta-bound}
\|\phi(t,t_{0},x,u)\|_{\mathcal{X}}\leq\beta(\|x\|_{\mathcal{X}},t-t_{0})+ \eta \mbox{, $\forall t\in[t_{0},t_{0}+T]$},
\end{align}
where $\beta\in \KL$ is the function given by the definition of $0$-GUAS. 
\end{lema}
The proof of Lemma~\ref{lem:main} is given in Appendix~\ref{app:proof-lem:main}.

We are now ready to provide the proof of our main result.
\begin{proof}[{\bf Proof of Theorem \ref{thm:main}}]
(\ref{item:b}) $\Rightarrow$ (\ref{item:a}) is straightforward. We next prove (\ref{item:a}) $\Rightarrow$ (\ref{item:b}). 

Assume (\ref{item:a}). We prove iISS using Theorem~\ref{thm:ISS-e-d} and taking into account that ISS means iISS in this case (Definition~\ref{def:iISS}) given that the admissible input functional satisfies condition (E). From Lemma \ref{lem:o-guas+UGR:UGS} we have that $\Sigma$ is UGS and therefore (Definition~\ref{def:iISS}) UBEBS0. Let $\alpha, \rho \in \Ki$ be the functions given by the definition of UGS. 

Let $T>0$, $r>0$ and $s>0$. Let $t_{0}\ge 0$, $x\in \X$ such that $\|x\|_{\X}\le r$ and $u\in\mathcal{U}$ with $\|u\|_{\U}\leq s$. Then, due to UGS we have that for all $t\in[t_{0},t_{0}+T]$,
\begin{eqnarray*}
\|\phi(t,t_{0},x,u)\|_{\mathcal{X}}&\leq &\alpha(\|x\|_{\mathcal{X}})+\rho\left(\|u\|_{\U}\right)\\
&\leq &\alpha(r)+\rho\left(s\right).
\end{eqnarray*}
Therefore, C\ref{item:ubrs}) holds with $C=\alpha(r)+\rho\left(s\right)$.

Let $\varepsilon>0$. Pick $\delta>0$ such that $\alpha(\delta)+\rho\left(\delta\right)\le\varepsilon$. Then, if $t_0\ge 0$, $x\in \X$ with  $\|x\|_{\mathcal{X}}\leq \delta$ and $u\in\mathcal{U}$ with $\|u\|_{\U}\leq \delta$, it follows that for all $t\ge t_0$
\begin{eqnarray*}
\|\phi(t,t_{0},x,u)\|_{\mathcal{X}}&\leq &\alpha(\|x\|_{\mathcal{X}})+\rho\left(\|u\|_{\U}\right)\\
&\leq &\alpha(\delta)+\rho\left(\delta\right)\\
&<&\varepsilon,
\end{eqnarray*}
and thus C\ref{item:smallx0u}) holds.

Next, we prove C\ref{item:attract}). Define $\nu\in\Ki$ via $\nu=2\rho$ and let $\psi=\rho^{-1}\comp\alpha$. Let $r\geq\varepsilon>0$, $t_{0}\ge 0$, $x\in \X$ be such that $\|x\|_{\mathcal{X}}\leq r$ and $u\in\mathcal{U}$. Distinguish the cases
\begin{enumerate}[(i)]
\item \label{item:i}$\|u\|_{\U}\ge\psi(r)$; and
\item \label{item:ii}$\|u\|_{\U}<\psi(r)$. 
\end{enumerate}
%
In case (\ref{item:i}), we have 
\begin{eqnarray*}
\|\phi(t,t_{0},x,u)\|_{\mathcal{X}}&\leq &\alpha(\|x\|_{\mathcal{X}})+\rho\left(\|u\|_{\U}\right)\\
&\leq &\alpha(r)+\rho\left(\|u\|_{\U}\right)\\
&\leq &\alpha(\psi^{-1}(\|u\|_{\U}))+\rho\left(\|u\|_{\U}\right)\\
&=&\rho\left(\|u\|_{\U}\right)+\rho\left(\|u\|_{\U}\right)\\
&=&2\rho(\|u\|_{\U})\\
&=&\nu(\|u\|_{\U})
\end{eqnarray*} 
So, for every $\varepsilon>0$ and $T>0$, it happens that $\|\phi(t,t_{0},x,u)\|_{\mathcal{X}}\leq\varepsilon+\nu(\|u\|_{\U})$ for every $t\geq t_{0}+T$.

In case (\ref{item:ii}), we have 
\begin{eqnarray*}
\|\phi(t,t_{0},x,u)\|_{\mathcal{X}}&\leq &\alpha(\|x\|_{\mathcal{X}})+\rho\left(\|u\|_{\U}\right)\\
&\leq &\alpha(r)+\rho\left(\|u\|_{\U}\right)\\
&\leq &\alpha(r)+\rho\left(\psi(r)\right)=\tilde{r} 
\end{eqnarray*}
So $\|\phi(t,t_{0},x,u)\|_{\mathcal{X}}\leq\tilde{r}$ for all $t\geq t_{0}$. Let $\tilde{\varepsilon}=\alpha^{-1}(\varepsilon)$ and $\eta=\tilde{\varepsilon}/2$. Pick $\tilde{T}>0$  such that $\beta(\tilde{r},\tilde{T})<\tilde{\varepsilon}/2$, where $\beta\in\KL$ is given by 0-GUAS. By Lemma~\ref{lem:main}, there exists $\gamma = \gamma(\tilde{r}, \eta, \tilde{T}) > 0$ such that (\ref{eq:beta-bound}) holds, with $\tilde{T}$ instead of $T$, provided that $\|u\|_{\U}\leq\gamma$. 

Define $N=\left\lceil\frac{\psi(r)}{\gamma}\right\rceil$ and $T=N\tilde{T}$, where $\lceil s \rceil$ denotes the smallest integer not less than $s\in\mathbb{R}$.

For $i=0,\ldots,N$, let $t_{i}=t_{0}+i\tilde{T}$. We consider the intervals $I_{i}=(t_{i},t_{i+1}]$ with $i=0,\ldots,N-1$ and claim that there exists an integer $j\leq N-1$ for which $\|u_{(t_{j},t_{j+1}]}\|_{\U}<\gamma$. If such a $j$ did not exist, then from the definition of $N$ and condition (E), it would follow that $\|u\|_{\U}\ge \|u_{(t_0,T]}\|_{\U}\ge \sum_{i=0}^{N-1}\|u_{(t_{i},t_{i+1}]}\|_{\U}\ge N\gamma \ge \psi(r)$, which contradicts case (\ref{item:ii}). 

Pick $j$ such that $\|u_{(t_{j},t_{j+1}]}\|_{\U}<\gamma$ and define $u_{j}=u_{(t_j,t_{j+1}]}$ and $x_j=\phi(t_j,t_0,x,u)$. 
By the causality and semigroup properties, $\phi(t,t_{0},x,u)=\phi(t,t_{j},x_j, u_{j})$ for all $t\in [t_{j},t_{j+1}]$. Since $\|\phi(t,t_{0},x,u)\|_{\mathcal{X}}\leq\tilde{r}$ for all $t\ge t_0$, we have that $\|\phi(t,t_{j},x_j,u_j)\|_{\mathcal{X}}\leq\tilde{r}$ for all $t\in [t_j,t_{j+1}]$. From the facts that $\|u_j\|_{\U}\le \gamma$ and the definition of $\gamma$ it follows that if $x_{j+1}=\phi(t_{j+1},t_{j},x_j,u_{j})$, then
\begin{eqnarray*}
\|x_{j+1}\|_{\mathcal{X}}&= &\|\phi(t_{j+1},t_{j},x_j,u_{j})\|_{\mathcal{X}}\\
&\leq &\beta(\|x_j\|_{\mathcal{X}},\tilde{T})+\eta\\
&\leq &\beta(\tilde{r},\tilde{T})+\eta\\
&\leq &\frac{\tilde{\varepsilon}}{2}+\frac{\tilde{\varepsilon}}{2}=\tilde{\varepsilon}
\end{eqnarray*}
Therefore, since $\phi(t,t_{0},x,u)=\phi(t,t_{j+1},x_{j+1}, u)$ for all $t\ge t_{j+1}$ and recalling the UGS property, it follows that for all $t\ge t_0+T \ge t_{j+1}$, 
\begin{eqnarray*}
\|\phi(t,t_{0},x,u)\|_{\mathcal{X}}&\leq &\alpha(\|x_{j+1}\|_{\mathcal{X}})+\rho\left(\|u\|_{\U}\right)\\
&\leq &\alpha(\tilde{\varepsilon})+\rho\left(\|u\|_{\U}\right)\\
&=&\varepsilon+\rho\left(\|u\|_{\U}\right)\\
&\le&\varepsilon+\nu\left(\|u\|_{\U}\right).
\end{eqnarray*}
This shows that C\ref{item:attract}) is satisfied. By Theorem~\ref{thm:ISS-e-d}, the system $\Sigma$ is $\|\cdot\|_{\U}$-ISS and hence $\|\cdot\|_{\U}$-iISS from Definition~\ref{def:iISS}.
\end{proof}

\section{Time-delay systems}
\label{sec:time-delay}

In this section, we consider time-delay systems with inputs. For $\tau\ge 0$ (where $\tau$ is larger than, or equal to, the maximum delay involved in the dynamics), let $\mathcal{C}=\mathcal{C}\left([-\tau,0],\mathbb{R}^{n}\right)$ be the set of continuous functions $\psi:[-\tau,0]\to \R^n$ endowed with the supremum norm $\|\psi\|=\displaystyle\sup\{|\psi(s)|:s\in [-\tau ,0]\}$. As usual, given a continuous function $x:[t_0-\tau,T)\to \R^n$ and any $t_0\le t<T$, $x_t$ is defined as the function $x_t:[-\tau,0]\to \R^n$ satisfying $x_{t}(s)=x(t+s)$ for all $s\in [-\tau,0]$, so that $x_t\in \mathcal{C}$.

Consider the system with inputs defined by the following retarded functional differential equation
\begin{align}\label{eq:sistr}
\dot{x}(t)=f(t,x_{t},u(t))
\end{align}
where $t\ge 0$, $x(t)\in \R^n$, $u(t) \in \R^m$ and $f:\R_{\geq 0}\times\mathcal{C}\times\R^{m}\to\R^{n}$. In this section, $\U$ denotes the set of all the functions $u:[0,\infty)\to \R^m$ that are locally bounded and Lebesgue measurable.

We assume that $f(t,\cdot,\cdot)$ is continuous for every $t\ge 0$, that $f(\cdot,\psi,\mu)$ is Lebesgue measurable for every $(\psi,\mu)\in \mathcal{C}\times \R^m$, and that for every $t_0\ge 0$, $\psi \in \mathcal{C}$ and $u\in \U$, there exists a unique maximally defined continuous function $x:[t_0-\tau,t_{(t_0,\psi,u)})\to \R^n$, with $t_{(t_0,\psi,u)}>t_0$ and $x_{t_0}=\psi$, that is locally absolutely continuous on $[t_0,t_{(t_0,\psi,u)})$ and satisfies equation (\ref{eq:sistr}) for almost all $t\in [t_0,t_{(t_0,\psi,u)})$.

Under these assumptions, take $\X=\mathcal{C}$, $\|\cdot\|_{\X} = \|\cdot\|$ and define the map $\phi:D_{\phi}\to \X$, with $D_{\phi}=\{(t,s,\psi,u)\in \R_{\ge 0}\times \R_{\ge 0} \times \mathcal{C} \times \U : s\le t <t_{(s,\psi,u)}\}$ and $\phi(t,s,\psi,u)=x_t$, where $x:[s-\tau,t_{(s,\psi,u)})\to \R^n$ is the unique maximally defined solution of (\ref{eq:sistr}) corresponding to the initial time $s$, the initial state $\psi$ and input $u$. Then, $\Sigma^{R}=(\X,\U,\phi)$ is a system as per Definition~\ref{def:system}.

For a system of the form (\ref{eq:sistr}), the $0$-GUAS, UBEBS, and iISS properties are usually defined as follows (see e.g. \cite{chagok_tac21}).
\begin{defin}
  \label{def:delay-stab}
  The time-delay system (\ref{eq:sistr}) is:
\begin{enumerate}
\item $0$-GUAS if there exists $\beta \in \KL$ such that the solution $x(\cdot)$ corresponding to any $t_0\ge 0$, $\psi \in \C$ and $u=\zi$ satisfies
\begin{align} \label{eq:a}
|x(t)|\le \beta(\|\psi\|,t-t_0)\quad \forall t\ge t_0;
\end{align}
\item UBEBS if there exist $\alpha, \rho, \kappa \in \Ki$ and $c\ge 0$ such that 
\begin{align} \label{eq:b}
{\color{black} |x(t)|\le \alpha(\|\psi\|)+\rho(\|u_{(t_0,t]}\|_{\kappa})+c\quad \forall t\ge t_0;}
\end{align}
\item iISS if there exist $\beta \in \KL$ and $\rho,\kappa\in\Ki$ such that
\begin{align} \label{eq:c}
{\color{black} |x(t)|\le \beta(\|\psi\|,t-t_0)+ \rho(\|u_{(t_0,t]}\|_{\kappa})\quad \forall t\ge t_0.}
\end{align}
\end{enumerate}
In~(\ref{eq:b}) and~(\ref{eq:c}), $x(\cdot)$ is the solution corresponding to initial time $t_0\ge 0$, initial state $\psi\in \C$ and input $u\in \U$, and $\kappa$ is referred to as the UBEBS or iISS gain, respectively.
\end{defin}

These definitions are equivalent to those corresponding to Definitions~\ref{def:stability-def} and~\ref{def:iISS}, as we next show.
\begin{prop}
  \label{prop:delay-general-equiv}
  Consider a time-delay system of the form~(\ref{eq:sistr}) and its corresponding system $\Sigma^R$ as defined above. Then,
  \begin{enumerate}[a)]
  \item System~(\ref{eq:sistr}) is 0-GUAS as per Definition~\ref{def:delay-stab} $\Leftrightarrow$ $\Sigma^R$ is 0-GUAS as per Definition~\ref{def:stability-def}.\label{item:td1}
  \item System~(\ref{eq:sistr}) is UBEBS as per Definition~\ref{def:delay-stab} $\Leftrightarrow$ $\Sigma^R$ is UBEBS as per Definitions~\ref{def:iISS} and~\ref{def:stability-def}.\label{item:td2}
  \item System~(\ref{eq:sistr}) is iISS as per Definition~\ref{def:delay-stab} $\Leftrightarrow$ $\Sigma^R$ is iISS as per Definitions~\ref{def:iISS} and~\ref{def:stability-def}. \label{item:td3}
  \end{enumerate}
\end{prop}
\begin{proof}
The {\em if} parts are a direct consequence of the fact that $|x(t)| \le \|x_t\|$. We next prove the {\em only if} parts. The only if part of item~\ref{item:td2}) is also straightforward, since if (\ref{eq:b}) holds, then the same equation holds with $\|x_t\|$ instead of $|x(t)|$ and with the function $\tilde \alpha$, defined by $\tilde{\alpha}(r)=\alpha(r)+r$, in place of $\alpha$.

Suppose that system~(\ref{eq:sistr}) is iISS as per Definition~\ref{def:delay-stab} and let $\beta \in \KL$ and $\rho\in \Ki$ be as in~(\ref{eq:c}). Without loss of generality we can suppose that $\beta(r,0)\ge r$ for all $r\ge 0$. By Sontag's Lemma on $\KL$-functions \cite[Prop.~7]{sontag_scl98}, there exist $\alpha_1,\alpha_2\in \Ki$ such that $\beta(r,t)=\alpha_2(\alpha_1(r) e^{-t})$ for all $r,t\ge 0$. Define $\tilde \beta(r,t)=\alpha_2(e^{\tau}\alpha_1(r) e^{-t})$, then $\tilde \beta \in \KL$ and $\beta\le \tilde \beta$. Suppose that $x(\cdot)$ satisfies (\ref{eq:c}). If $t\ge t_0+\tau$, then for all $s\in [-\tau,0]$
\begin{align*}
|x(t+s)|&\le \beta(\|\psi\|,t+s-t_0)+\rho(\|u\|_{\kappa})\\
&\le \beta(\|\psi\|,t-t_0-\tau)+\rho(\|u\|_{\kappa})\\
&\le \alpha_2(e^{\tau}\alpha_1(\|\psi\|) e^{-(t-t_0)})+\rho(\|u\|_{\kappa})\\
&\le \tilde \beta(\|\psi\|,t-t_0)+\rho(\|u\|_{\kappa}).
\end{align*}
Hence $\|x_t\|\le \tilde \beta(\|\psi\|,t-t_0)+\rho(\|u\|_{\kappa})$ for all $t\ge t_0+\tau$. If $t_0\le t < t_0+\tau$, for all $s\in [-\tau,0]$
\begin{align*}
|x(t+s)|&\le \beta(\|\psi\|,0)+\rho(\|u\|_{\kappa})\\
&\le \alpha_2(\alpha_1(\|\psi\|))+\rho(\|u\|_{\kappa})\\
&\le \alpha_2(e^{\tau}\alpha_1(\|\psi\|)e^{-(t-t_0)})+\rho(\|u\|_{\kappa})\\
&\le \tilde \beta(\|\psi\|,t-t_0)+\rho(\|u\|_{\kappa}).
\end{align*}
In this case, we have that $\|x_t\|\le \tilde \beta(\|\psi\|,t-t_0)+\rho(\|u\|_{\kappa})$ for all $t_0\le t< t_0+\tau$. Thus $\Sigma^R$ is iISS as per Definitions~\ref{def:iISS} and~\ref{def:stability-def}.

The only if part of item~\ref{item:td1}) can be proved in the same way.
\end{proof}

Assumption~\ref{as:f-delay} gives sufficient conditions on the function $f$ in (\ref{eq:sistr}) for iISS to be equivalent to $0$-GUAS $\wedge$ UBEBS.
\begin{as}
  \label{as:f-delay}
  The function $f$ in (\ref{eq:sistr}) satisfies the following conditions.
  \begin{enumerate}
  \item[(R1)] \label{item:r1} There exists $\gamma\in\Ki$ and $N:\mathbb{R}_{\geq 0}\rightarrow\mathbb{R}_{> 0}$ non-decreasing such that
    \begin{equation} \label{eq:r1}
      |f(t,\psi,\mu)|\leq N(\|\psi\|)\left(1+\gamma(|\mu|)\right)
    \end{equation}
    for all $t\geq 0$, for every $\psi\in\mathcal{C}$ and for all $\mu\in\mathbb{R}^{m}$.
  \item[(R2)] \label{item:r2} For every $r>0$ and $\varepsilon >0$, there exists $\delta >0$ such that for all $t\geq 0$, it is true that 
    \begin{equation*}
      |f(t,\psi,\mu)-f(t,\psi,0)|<\varepsilon
    \end{equation*}
    if $\|\psi\|\leq r$ and $|\mu|\leq\delta$.
  \item[(R3)] \label{item:r3} $f(t,\psi,0)$ is Lipschitz in $\psi$ on bounded sets, uniformly in $t\ge 0$, i.e., for all $r>0$ there exists $L=L(r)$ such that $|f(t,\psi,0)-f(t,\varphi,0)|\leq L\|\psi-\varphi\|$ for all $t\ge 0$ whenever $\|\psi\| \le r$ and $\|\varphi\| \le r$.
\end{enumerate}
\end{as}
{\color{black} \begin{rem} \label{rem:bound} Condition (R1) is equivalent to the existence of $\hat \gamma\in\Ki$ and $\hat N:\mathbb{R}_{\geq 0}\rightarrow\mathbb{R}_{> 0}$ non-decreasing such that
    \begin{equation} \label{eq:r1p}
      |f(t,\psi,\mu)|\leq \hat N(\|\psi\|)+\hat \gamma(|\mu|)
    \end{equation}
    for all $t\geq 0$, for every $\psi\in\mathcal{C}$ and for all $\mu\in\mathbb{R}^{m}$. This is because if (\ref{eq:r1}) holds, then using the fact that $N\gamma\le (N^2+\gamma^2)/2$ it follows that (\ref{eq:r1p}) holds with $\hat N(r)=N(r)+\frac{N(r)^2}{2}$ and $\hat \gamma(r)=\frac{\gamma(r)^{2}}{2}$. Conversely, if (\ref{eq:r1p}) holds, then (\ref{eq:r1}) holds with $N(r)=\max\left\{\hat N(r),\frac{ \hat N(r)}{\hat N(0)} \right\}$ and $\gamma(r)=\hat \gamma(r)$, because $\hat N(0) > 0$ and $\hat N(r)/\hat N(0) \ge 1$.
\end{rem}
    }
    
The following lemma, whose proof can be obtained, {\em mutatis mutandis}, from that of Lemma~1 in \cite{haiman_tac18}, asserts that Assumption~\ref{as:f-delay} holds if $f(t,0,0)=0$ for all $t\ge 0$ and $f$ satisfies a Lipschitz condition on bounded sets.
\begin{lema} \label{lem:f-lips-delay} Suppose that $f:\R_{\ge 0}\times \C \times \R^m\to \R^n$ is Lipschitz on bounded subsets of $\C \times \R^m$, uniformly in $t$, i.e. for all $r\ge 0$ there exists $L=L(r)\ge 0$ such that for all $\psi,\theta \in \C$ such that $\|\psi\|\le r$ and $\|\theta\|\le r$ and all $\mu,\nu\in \R^m$ with $|\mu|\le r$ and $|\nu|\le r$ we have that
\begin{align*}
|f(t,\psi,\mu)-f(t,\theta,\nu)|\le L(\|\psi-\theta\|+|\mu-\nu|)\quad \forall t\ge 0.
\end{align*}
Suppose in addition that $f(t,0,0)=0$ for all $t\ge 0$. Then $f$ satisfies Assumption \ref{as:f-delay}.
\end{lema}
\begin{teo}
  \label{thm:iISS-ubebs+0-guas}
  Consider system (\ref{eq:sistr}) and let Assumption \ref{as:f-delay} hold. Let $\gamma\in \Ki$ be given by (R1). Then, the following hold.
  \begin{enumerate}
  \item[a)] If system (\ref{eq:sistr}) is iISS with gain $\kappa$, then it is 0-GUAS and UBEBS with gain $\kappa$.
    
  \item[b)] If system (\ref{eq:sistr}) is $0$-GUAS and UBEBS with gain $\alpha$, then it is iISS with gain $\kappa=\max\{\alpha,\gamma\}$.
  \end{enumerate}
\end{teo}
The proof of Theorem \ref{thm:iISS-ubebs+0-guas} is a consequence of Theorem~\ref{thm:main} and the following lemma.
\begin{lema}
  \label{lem:teosigrh}
  Let Assumption~\ref{as:f-delay} hold and let $\gamma\in \K$ be given by (R1). Then, system $\Sigma^{R}$ satisfies Assumption~\ref{as:transition-map} with $\|\cdot\|_{\U}=\|\cdot\|_{\gamma}$.
\end{lema}
The proof of Lemma~\ref{lem:teosigrh} is provided in Appendix \ref{app:p-ds}.

\begin{proof}[{\bf Proof of Theorem \ref{thm:iISS-ubebs+0-guas}}]
  Part a) is straightforward; we next prove b). 

  Assume that (\ref{eq:sistr}) is $0$-GUAS and UBEBS with gain $\alpha$. Let $\kappa=\max\{\alpha,\gamma\}\in \Ki$. Then, (\ref{eq:sistr}) is also UBEBS with gain $\kappa$ because $\|u\|_{\alpha}\le \|u\|_{\kappa}$ for all $u\in \U$. By Proposition~\ref{prop:delay-general-equiv}, $\Sigma^R$ is $\|\cdot\|_{\kappa}$-UBEBS and $0$-GUAS. From Lemma~\ref{lem:teosigrh}, $\Sigma^R$ satisfies Assumption~\ref{as:transition-map} with $\|\cdot\|_{\U} = \|\cdot\|_{\gamma}$, and hence also with $\|\cdot\|_{\U}=\|\cdot\|_{\kappa}$. By Theorem~\ref{thm:main}, $\Sigma^R$ is then $\|\cdot\|_{\kappa}$-iISS and, by Proposition~\ref{prop:delay-general-equiv}, (\ref{eq:sistr}) is iISS with gain $\kappa$.
\end{proof}
{\color{black} The equivalence between $0$-GUAS $\wedge$ UBEBS and iISS has been proved recently in \cite{chagok_tac21} ( see a) $\Leftrightarrow$ e) in Theorem 2 of \cite{chagok_tac21} ) for time-invariant time-delay systems under the stronger assumption that the function $f(x_t,u)$ is Lipschitz on bounded subsets of $\C \times \R^m$ \cite[Standing assumption~1]{chagok_tac21}. The proof of $0$-GUAS $\wedge$ UBEBS implying iISS in \cite{chagok_tac21} is based on the existence of a time-invariant, Lipschitz on bounded subsets and coercive Lyapunov-Krasovskii functional (LKF) $V$ for the zero-input system $f(x_t,0)$ \cite{pepkar_ijc13} and uses the Lipschitz condition on $f$ in an essential way ( see the proof of i) $\Rightarrow$ ii) in \cite[Proposition~3]{chagok_tac21} ).
In view of Lemma~\ref{lem:f-lips-delay}, the equivalence a) $\Leftrightarrow$ e) in \cite[Thm. 2]{chagok_tac21} becomes then a corollary of Theorem~\ref{thm:iISS-ubebs+0-guas}, but the assumptions of Theorem~\ref{thm:iISS-ubebs+0-guas} particularized to the case of time-invariant time-delay systems are clearly weaker than those of \cite[Thm. 2]{chagok_tac21}.}

  By simplifying the analysis of iISS into the separate evaluation of 0-GUAS and UBEBS, Theorem \ref{thm:iISS-ubebs+0-guas} also allows to more easily conclude that if the function $f$ in~(\ref{eq:sistr}) is time-invariant and Lipschitz on bounded subsets, then the existence of an iISS LKF with pointwise dissipation (as per \cite{chagok_tac21}) implies that the time-delay system is iISS, which is one of the important results in \cite{chagok_tac21}. Moreover, Theorem \ref{thm:iISS-ubebs+0-guas} shows that this implication still holds for time-invariant systems satisfying the weaker Assumption~\ref{as:f-delay}, {\color{black} with the derivative of $V$ considered in Dini's sense.}

\section{Semilinear systems}
\label{sec:semilinear-systems}
In this section, we apply our main result to obtain a characterization of iISS for a semilinear system of the form
\begin{align}
  \label{eq:sistsl}
  \begin{split}
    \dot{x}(t) &=Ax(t)+f(t,x(t),u(t))  \\
    x(t_{0}) &=x_{0} 
  \end{split}
\end{align}
where $t\ge 0$, $x(t)\in \X$, $\X$ a Banach space with norm $\|\cdot\|_{\X}$, $u(t)\in \Ui$, with $\Ui$ a normed space with norm $\|\cdot\|_{\Ui}$. The operator $A:D(A)\subseteq\mathcal{X}\rightarrow\mathcal{X}$ is a linear operator that generates a strongly continuous semigroup (a $C_{0}$-semigroup) $T:\R_{\ge 0} \to \mathcal{L}(\X)$, where $\mathcal{L}(\X)$ is the set of all the linear and bounded operators from $\X$ to $\X$,   and $f:\R_{\ge 0}\times \X \times \Ui\to \X$. 
The set $\U$ of admissible inputs is the set of all the piecewise continuous functions $u:\R_{\ge 0} \to \Ui$. 

Given $t_0\ge 0$, $x_0\in \X$ and $u\in \U$, consider the weak solutions of (\ref{eq:sistsl}). A function $x:J\to \X$, with $J=[t_0,\tau)$ or $[t_0,\tau]$ is a weak solution of (\ref{eq:sistsl}) if it is continuous and 
\begin{align*}
x(t)=T(t-t_0)x_0+\int_{t_0}^t T(t-s)f(s,x(s),u(s))\:ds, \quad \forall t\in J
\end{align*}
where the concept of integral is that of Bochner \cite{cazenave_1998}.

\subsection{Semilinear systems: general results}
\label{sec:semilinear-general}

The following assumptions on $f$ are required. 
\begin{as} \label{ass:f-sml}
  The function $f$ in (\ref{eq:sistsl}) satisfies the following conditions.
  \begin{enumerate}
  \item[(SL1)] $f$ is piecewise continuous in $t$ and continuous in its other variables in the following sense. There exists a strictly increasing and unbounded sequence of positive times $\{\tau_k\}_{k=1}^{\infty}$ and continuous functions $f_k:[\tau_k,\tau_{k+1}]\times \X \times \Ui \to \X$, $k=0,1,\ldots$ with $\tau_0=0$, such that $f=f_k$ on $[\tau_k,\tau_{k+1})\times \X \times \Ui$.
  \item[(SL2)] $f(t,\xi,\mu)$ is Lipschitz in $\xi$ on bounded sets, uniformly for all $t$ and for $\mu$ in bounded sets, i.e., for all $r>0$ there exists $L=L(r)\ge 0$ such that, for all $\xi,\omega \in \mathcal{X}$ such that $\|\xi\|_{\mathcal{X}}\leq r$, $\|\omega \|_{\mathcal{X}}\leq r$, all $\mu\in\Ui$ such that $\|\mu\|_{\Ui}\leq r$ and all $t\ge 0$, it holds that
    \begin{equation*}
      \|f(t,\xi,\mu)-f(t,\omega,\mu)\|_{\mathcal{X}}\leq L\|\xi-\omega\|_{\mathcal{X}}.
    \end{equation*}
  \item[(SL3)] \label{item:sl3} There exists $\gamma\in\Ki$ and $N:\mathbb{R}_{\geq 0}\rightarrow\mathbb{R}_{> 0}$ non-decreasing such that
    \begin{equation*}
      \|f(t,\xi,\mu)\|_{\X}\leq N(\|\xi\|_{\X})\left(1+\gamma(\|\mu\|_{\Ui})\right)
    \end{equation*}
    for all $t\geq 0$, $\xi\in\X$ and $\mu\in\Ui$.
  \item[(SL4)] \label{item:sl4} For every $r>0$ and $\varepsilon >0$, there exists $\delta >0$ such that for all $t\geq 0$, it is true that 
    \begin{equation*} 
      \|f(t,\xi,\mu)-f(t,\xi,0)\|_{\X}<\varepsilon
    \end{equation*}
    if $\|\xi\|_{\X}\leq r$ and $\|\mu\|_{\Ui}\leq\delta$.
  \end{enumerate}
\end{as}
{\color{black} Condition (SL3) can be replaced by an equivalent condition which is analogous to that appearing in (\ref{eq:r1p}) in Remark \ref{rem:bound}.}

When $f(t,\xi,\mu)$ is Lipschitz in $(\xi,\mu)$ on bounded sets and satisfies $f(t,0,0)\equiv 0$, it can be proved, similarly to the proof of Lemma~\ref{lem:f-lips-delay}, that $f$ satisfies (SL2)--(SL4) of Assumption~\ref{ass:f-sml}. This is made more precise as follows.
\begin{lema}
  \label{lem:f-lips-sl}
  Suppose that $f:\R_{\ge 0}\times \X \times \Ui \to \X$ is Lipschitz on bounded subsets of $\X \times \Ui$ uniformly over $\R_{\ge 0}$, i.e. for all $r\ge 0$ there exists $L=L(r)\ge 0$ such that for all $\xi,\zeta \in \X$ such that $\|\xi\|_{\X}\le r$ and $\|\zeta\|_{\X}\le r$ and all $\mu, \nu \in \Ui$ with $\|\mu\|_{\Ui}\le r$ and $\|\nu\|_{\Ui}\le r$ we have that
\begin{align*}
\|f(t,\xi,\mu)-f(t,\zeta,\nu)\|_{\X}\le L(\|\xi-\zeta\|_{\X} + \|\mu-\nu\|_{\Ui}) \quad \forall t\ge 0.
\end{align*}
Suppose in addition that $f(t,0,0)=0$ for all $t\ge 0$. Then $f$ satisfies (SL2)--(SL4) of Assumption~\ref{ass:f-sml}.
\end{lema}
Under (SL1)--(SL3) of Assumption \ref{ass:f-sml} and the fact that the admissible inputs $u$ are piecewise continuous, a slight modification of \cite[Prop. 4.3.3]{cazenave_1998} to allow piecewise continuity proves that for every $t_0\ge 0$, $x_0\in \X$ and $u\in \U$ there exists a unique maximally defined weak solution $x:[t_0, t_{(t_0,x_0,u)})\to \X$ of (\ref{eq:sistsl}).

Defining the map $\phi:D_{\phi}\to \X$, with $D_{\phi}=\{(t,t_0,x_0,u)\in \R_{\ge 0}\times \R_{\ge 0} \times \X \times \U:t_0\le t<t_{(t_0,x_0,u)}\}$ and $\phi(t,t_0,x_0,u)=x(t)$ with $x:[t_0,t_{(t_0,x_0,u)})\to \X$ the unique maximally defined weak solution of (\ref{eq:sistsl}), we have that $\Sigma^{SL}=(\X,\U,\phi)$, which will be referred to as the system generated by~(\ref{eq:sistsl}), is a system according to Definition~\ref{def:system}. 

{\color{black} Under (SL1)-(SL2) of Assumption \ref{ass:f-sml}, system $\Sigma^{SL}$ has the boundedness implies continuation (BIC) property, as the following Lemma shows. The proof is given in Appendix \ref{app:BIC}.
\begin{lema} \label{lem:BIC}  Consider the semilinear system~(\ref{eq:sistsl}) and let (SL1)-(SL2) of Assumption~\ref{ass:f-sml} hold. Then $\Sigma^{SL}$ has the BIC property.
\end{lema}
}
The following Lemma asserts that $\Sigma^{SL}$ satisfies Assumption~\ref{as:transition-map} if Assumption~\ref{ass:f-sml} holds. The proof is provided in Appendix~\ref{app:p-sls}.
\begin{lema}
  \label{lem:lema-ass2-ssl}
  Consider the semilinear system~(\ref{eq:sistsl}), let Assumption~\ref{ass:f-sml} hold, let $\gamma\in \Ki$ be the function from (SL3), and let $\Sigma^{SL}$ be the system generated by~(\ref{eq:sistsl}). Then, $\Sigma^{SL}$ satisfies Assumption~\ref{as:transition-map} with $\|\cdot\|_{\U}=\|\cdot\|_{\gamma}$.
\end{lema}
The following characterization of iISS can be proved almost identically as Theorem~\ref{thm:iISS-ubebs+0-guas}, but invoking Lemma~\ref{lem:lema-ass2-ssl} instead of Lemma~\ref{lem:teosigrh}.
\begin{teo}\label{thm:iISS-ubebs+0-guas-sl}
Let $\Sigma^{SL}$ be the system generated by equation (\ref{eq:sistsl}). Suppose that Assumption \ref{ass:f-sml} holds and let $\gamma\in \Ki$ be the function coming from (SL3) of such assumption. Then the following hold.
\begin{enumerate}
\item[a)] If system $\Sigma^{SL}$ is iISS with iISS-gain $\kappa$, then $\Sigma^{SL}$ is 0-GUAS and UBEBS with gain $\kappa$.

\item[b)] If system $\Sigma^{SL}$ is $0$-GUAS and UBEBS with UBEBS-gain $\alpha$, then $\Sigma^{SL}$ is iISS with gain $\kappa=\max\{\alpha,\gamma\}$.
\end{enumerate}
\end{teo}

\subsection{Semilinear systems: generalized bilinear form}
\label{sec:gener-bilin}

The much stronger result that 0-GUAS on its own is equivalent to iISS (Theorem~\ref{thm:bilineal}) can be obtained when Assumption~\ref{ass:f-sml} is replaced by the following stronger condition, which replaces (SL3)--(SL4) by a bound on $\|f(t,\xi,\mu)\|_{\X}$ of a specific, affine-in-$\|\xi\|_{\X}$ form.
\begin{as}
  \label{as:bilinear}
  Let $f$ in~(\ref{eq:sistsl}) satisfy (SL1) and (SL2) of Assumption~\ref{ass:f-sml}, jointly with the bound
\begin{align} \label{eq:bound-bi}
\|f(t,\xi,\mu)\|_{\X}\le (K\|\xi\|_{\X}+d)\gamma(\|\mu\|_{\Ui})\quad \forall \xi\in \X, \mu\in \Ui
\end{align}
for some constants $K,d\in \R_{\ge 0}$ and some $\gamma \in \Ki$. 
\end{as}
If a nonlinear function $f$ satisfies Assumption~\ref{as:bilinear}, then (SL3) of Assumption~\ref{ass:f-sml} holds with $N(r)=Kr+d$ and the same function $\gamma$, and (SL4) follows directly from (\ref{eq:bound-bi}), since $f(t,\xi,0)=0$ for all $(t,\xi) \in \R_{\ge 0} \times \X$.

Examples of functions satisfying Assumption~\ref{as:bilinear} are those of the form $f(t,\xi,\mu)=B(t)\mu+C(t,\xi,\mu)$, where $B:\R_{\ge 0}\to \mathcal{L}(\Ui,\X)$ is piecewise continuous, $C:\R_{\ge 0}\times \X\times \Ui\to \X$ is piecewise continuous in $t$ and $C(t,\cdot,\cdot)$ is bilinear, and there exist constants $K,d\ge 0$ such that $\|B(t)\|\le d$ and $\sup_{\|x\|_{\X}= 1,\|u\|_{\Ui}= 1}\|C(t,\xi,\mu)\|_{\X}\le K$ for all $t\ge 0$. In this case (\ref{eq:bound-bi}) is satisfied with these values of $K$ and $d$, and with $\gamma(r)=r$.

Recall that the semigroup $T(\cdot)$ generated by the operator $A$ is exponentially stable if $\|T(t)\|\le Me^{-\lambda t}$ for some $M\ge 1$ and $\lambda>0$, where $\|T(t)\|$ denotes the induced norm of the operator $T(t)$. Also, exponential stability of $T(\cdot)$ is equivalent to GUAS of the system $\dot{x}=Ax$ \cite[Prop. 3]{dasmir_mcss13}.  

\begin{teo}
  \label{thm:bilineal}
  Consider a semilinear system (\ref{eq:sistsl}) that satisfies Assumption \ref{as:bilinear}. Then, the following are equivalent. 
\begin{enumerate}[a)]
\item \label{item:bilin-iiss}System (\ref{eq:sistsl}) is iISS.
\item \label{item:bilin-0guas}System (\ref{eq:sistsl}) is $0$-GUAS.
\end{enumerate}
\end{teo}
\begin{proof} Since \ref{item:bilin-iiss}) $\Rightarrow$ \ref{item:bilin-0guas}) is trivial, we prove \ref{item:bilin-0guas}) $\Rightarrow$ \ref{item:bilin-iiss}). Suppose that the system is $0$-GUAS. Then the semigroup $T(\cdot)$ is exponentially stable. Let $M\ge 1$ and $\lambda>0$ so that $\|T(t)\|\le Me^{-\lambda t}$ for all $t\ge 0$, where $\|T(t)\|$ denotes the induced norm of the operator $T(t)$. In the remainder of this proof, we omit the subscripts in the norms $\|\cdot\|_{\X}$ and $\|\cdot\|_{\Ui}$ in order to avoid cluttered notation and because these can be inferred from the context. 

Let $t_0\ge 0$, $x_0\in \X$, $u \in \U$ and $x(\cdot)$ be the corresponding trajectory. Let $[t_0,t_{(t_0,x_0,u)})$ be the maximal interval of definition of $x(\cdot)$. Suppose without loss of generality that $\|u\|<\infty$. Then, for all $t_0\le t < t_{(t_0,x_0,u)}$,
\begin{align*}
x(t)=T(t-t_0)x_0+\int_{t_0}^t T(t-s)f(s,x(s),u(s))ds.
\end{align*}
Take the norm at both sides of the equality and apply the triangle inequality and the properties of the norm of the integral to obtain
\begin{align*}
\|x(t)\|&\le\|T(t-t_0)\|\|x_0\|+\int_{t_0}^t \|T(t-s)\|\|f(s,x(s),u(s))\|ds \\
&\le M e^{-\lambda(t-t_0)}\|x_0\| + \int_{t_0}^t M e^{-\lambda(t-s)} (K\|x(s)\|+d)\gamma(\|u(s)\|) ds.
\end{align*}
Multiply both sides by $e^{\lambda (t-t_0)}$ and define $z(t)=e^{\lambda (t-t_0)}\|x(t)\|$, so that
\begin{align*}
z(t)&\le M \|x_0\|+ M d \int_{t_0}^t e^{\lambda(s-t_0)}\gamma(\|u(s)\|) ds+ \int_{t_0}^t MK \gamma(\|u(s)\|) z(s)  ds.
\end{align*}
Then, for $t_0\le t\le \tau <t_{(t_0,x_0,u)}$
\begin{align*}
z(t)&\le M \|x_0\|+ M d \int_{t_0}^{\tau} e^{\lambda (\tau-t_0)}\gamma(\|u(s)\|) ds+ MK \int_{t_0}^t \gamma(\|u(s)\|) z(s)  ds.
\end{align*}
By applying Gronwall's Lemma on the interval $[t_0,\tau]$ it follows that
\begin{align*}
z(\tau)\le M \left [ \|x_0\|+ d \int_{t_0}^{\tau}  e^{\lambda (\tau-t_0)} \gamma(\|u(s)\|) ds \right ]e^{MK \int_{t_0}^{\tau} \gamma(\|u(s)\|)ds}.
\end{align*}
Recalling the definition of $z$, multiplying both sides by $e^{-\lambda(\tau-t_0)}$,  and taking into account that $\int_{t_0}^{\tau} \gamma(\|u(s)\|)ds \le \|u\|_{\gamma}$ and $e^{-\lambda(\tau-t_0)} \le 1$ for all $\tau\ge t_0$, then also for all $\tau \in [t_0,t_{(t_0,x_0,u)})$ we have
\begin{align*}
\|x(\tau)\| &\le M [ \|x_0\|+ d \|u\|_{\gamma} ] e^{MK \|u\|_{\gamma}}\\
&\le \frac{\|x_0\|^2}{2}+ \frac{M^2e^{2MK \|u\|_{\gamma}}}{2}+ Md \|u\|_{\gamma} e^{MK \|u\|_{\gamma}}
\end{align*}
where we have used the fact that $ab\le \frac{a^2}{2}+\frac{b^2}{2}$ setting $a=\|x_0\|$ and $b= Me^{MK\|u\|_{\gamma}}$. Defining $\alpha(r)=\frac{r^2}{2}$, $\rho(r)=\frac{M^2(e^{2MK r}-1)}{2}+ Md r e^{MK r}$ and $c=\frac{M^2}{2}$, we have that $\alpha, \rho \in \Ki$ and 
\begin{align*}
\|x(\tau)\|\le \alpha(\|x_0\|)+\rho(\|u\|_{\gamma})+c \quad \forall \tau \in [t_0,t_{(t_0,x_0,u)}).
\end{align*}

Since $\Sigma^{SL}$ has the BIC property according to Lemma \ref{lem:BIC}, then $t_{(t_0,x_0,u)}=\infty$ and the corresponding system $\Sigma^{SL}$ is UBEBS as per Definitions \ref{def:stability-def} and \ref{def:iISS}. The iISS of the system then follows from Theorem \ref{thm:iISS-ubebs+0-guas-sl}.
\end{proof}

Theorem~\ref{thm:bilineal} generalizes \cite[Theorem~4.2]{mirito_mcrf16} to the time-varying case. The proof given here is based on the general characterization~(\ref{eq:iISS-superp}), while that in \cite{mirito_mcrf16} uses an {\em ad hoc} method.
A recent result dealing with the relationship between ISS and iISS for generalized bilinear time-invariant systems, allowing for unbounded (linear) input operators is given in \cite{hosjac_mcss22}. The results in the current paper are neither a special case nor more general than those of \cite{hosjac_mcss22}.

\section{Conclusions}
\label{sec:conclusions}
The equivalence between integral input-to-state stability (iISS) and the combination of global uniform asymptotic stability under zero input (0-GUAS) with uniformly bounded-energy input/bounded state (UBEBS) was established for systems defined in abstract form, provided a reasonable assumption of continuity of the trajectories with respect to the input, at the zero input, is satisfied and employing a more general definition of iISS. Sufficient conditions for this assumption to be satisfied were given for time-delay systems and for semilinear evolution equations over Banach spaces. The abstract definition of system employed allows for time-varying infinite-dimensional systems whose solutions are unique. It is expected that our main result could be helpful in (a) establishing the equivalence for other specific classes of infinite-dimensional systems, such as semilinear systems over Banach spaces involving unbounded input operators, for which very few results are currently available, and (b) giving mild conditions under which ISS implies iISS, as done for finite-dimensional systems in \cite{haiman_auto18}. {\color{black} Future work could also address the generalization of the asymptotic characterizations of iISS that involve some limit inferior of the norm of the trajectory, as per the BEFBS and BEWCS properties in \cite[Section~4.2]{anging_siamjco04}.}

\appendix
{\color{black} \section{Proof of Lemma \ref{lem:ass1}} \label{app:proof-lem-ass1}
In what follows we omit the subscripts in $\|\cdot\|_{\X}$ and $\|\cdot\|_{\U}$, which can be easily inferred from the context.

Suppose that the forward complete and UBRS system $\Sigma$ satisfies Assumption \ref{as:transition-map}. For $(t_0,x_0,u)\in \R_{\ge 0}\times \X \times \U$, let $\varphi(t,t_0,x_0,u)=\phi(t,t_0,x_0,u)-\phi(t,t_0,x_0,\zi)$ for all $t\ge t_0$.
For nonnegative $T,r$ and $s$ define
\begin{align}
\hat \Gamma(\ell,r,s)=\sup\{\|\varphi(t,t_0,x_0,u)\|:0\le t_0 \le t \le t_0+\ell,\:\|x_0\|\le r,\:\|u\|\le s\}.
\end{align}
The $UBRS$ of $\Sigma$ ensures that $\hat \Gamma(\ell,r,s)<\infty$. From the above definition, it follows that $\hat\Gamma$ is nondecreasing in each of its arguments and that for all $(t_0,x_0,u)\in \R_{\ge 0}\times \X \times \U$ such that $\|u\|<\infty$,
\begin{align} \label{eq:hatgamma}
\|\varphi(t,t_0,x_0,u)\|\le \hat \Gamma(t-t_0,\|x_0\|,\|u\|)\quad \forall t\ge t_0.
\end{align}
Next, we prove that $\lim_{s\to 0^+}\hat \Gamma(\ell,r,s)=0$ for all nonnegative $\ell$ and $r$.
Fix $\ell,r\ge 0$ and $\varepsilon>0$. Due to the UBRS property there exists $R=R(\ell,r)$ such that for all $t_0 \ge 0$, $x_0\in \X$ with $\|x_0\|\le r$ and $u\in \U$ with $\|u\|\le 1$ it follows that
\begin{align}
\max\{\|\phi(t,t_0,x_0,u)\|, \|\phi(t,t_0,x_0,\zi)\|\}\le R\quad \forall t\in [t_0,t_0+\ell]. 
\end{align}
Assumption \ref{as:transition-map} ensures the existence of $\delta>0$, which we can assume less than $1$, such that for all $t_0 \ge 0$, $x_0\in \X$ with $\|x_0\|\le r$ and $u\in \U$ with $\|u\|\le \delta$ we have 
\begin{align}
\|\varphi(t,t_0,x_0,u)\|\le \varepsilon \quad \forall t\in [t_0,t_0+\ell].
\end{align}
 In consequence, from the definition of $\hat \Gamma$ we have that $0\le \hat \Gamma(\ell,r,s)\le \hat \Gamma(\ell,r,\delta)\le \varepsilon$ for all $0\le s\le \delta$, and then that $\lim_{s\to 0^+}\hat \Gamma(\ell,r,s)=0$ follows. By standard arguments one can prove the existence of a function $\Gamma:\R^3_{\ge 0}\to \R_{\ge 0}$ which is continuous and strictly increasing in each of the first two arguments, $\Gamma(\ell,r,\cdot)\in \Ki$ for all $\ell$ and $r$ and $\Gamma \ge \hat \Gamma$. 
 From (\ref{eq:hatgamma}), the fact that $\Gamma\ge \hat \Gamma$ and causality it follows that
 \begin{align}
\|\varphi(t,t_0,x_0,u)\|\le \Gamma(t-t_0,\|x_0\|,\|u_{(t_0,t]}\|)\quad \forall t\ge t_0.
\end{align}
We then have proved that \ref{item:I}) implies \ref{item:II}).

That \ref{item:II}) implies \ref{item:I}) follows straightforwardly. Given $R,\varepsilon, T>0$, if there exists $t^* \in (t_0,t_0+T]$ such that both $\|\phi(t,t_0,x_0,u)\| \le R$ and $\|\phi(t,t_0,x_0,\zi)\| \le R$ are satisfied for all $t\in [t_0,t^*]$, then $\|x_0\|\le R$ and $\|\varphi(t,t_0,x_0,u)\| \le \Gamma(T,R,s)$ for all $t\in [t_0,t_0+T]$ provided that $\|u\|\le s$. Define $\bar\Gamma \in \Ki$ as $\bar\Gamma(s) = \Gamma(T,R,s)$ and set $\delta = \bar\Gamma^{-1}(\varepsilon)$. Then, $\|\varphi(t,t_0,x_0,u)\| \le \bar\Gamma(\delta) = \varepsilon$ for all $t\in [t_0,t_0+T]$ if $\|u\|\le\delta$. This shows that Assumption~\ref{as:transition-map} holds. 
}

  \section{Proof of Theorem \ref{thm:ISS-e-d}}
  \label{app:proof-thm:ISS-e-d}

Suppose that the system $\Sigma$ is $\|\cdot\|_{\U}$-ISS and let $\beta \in \KL$ and $\rho \in \Ki$ the functions characterizing this stability property. 

Let $T>0$, $r>0$ and $s>0$. Let $t_0\ge 0$, $x\in B_r$ and $u\in \U$ be such that $\|u\|_{\U}\leq s$. Then for all $t\ge t_0$
\begin{align*}
\|\phi(t,t_{0},x,u)\|_{\mathcal{X}}&\leq  \beta(\|x\|_{\mathcal{X}},t-t_{0})+\rho(\|u\|_{\U})\\
&\leq  \beta(r,t-t_{0})+\rho(\|u\|_{\U})\\
&\leq  \beta(r,0)+\rho(s).
\end{align*}
Thus $\Sigma$ satisfies C\ref{item:ubrs}) with $C=\beta(r,0)+\rho\left(s\right)$.

To prove C\ref{item:smallx0u}), take $\delta=\alpha^{-1}(\varepsilon)$ with $\alpha(\cdot)=\beta(\cdot,0)+\rho(\cdot)\in\Ki$. Indeed, if $t_0\ge 0$, $\|u\|_{\U}\leq\delta$ and $\|x\|_{\mathcal{X}}\leq\delta$ we have that
\begin{align*}
\|\phi(t,t_{0},x,u)\|_{\mathcal{X}}&\leq  \beta(\|x\|_{\mathcal{X}},t-t_{0})+\rho(\|u\|_{\U})\\
&\leq  \beta(\delta,t-t_{0})+\rho(\|u\|_{\U})\\
&\leq  \beta(\delta,0)+\rho\left(\delta\right)\\
&= \alpha(\delta)=\varepsilon.
\end{align*}

As for C\ref{item:attract}), let $0<\varepsilon\le r$. Since $\beta(r,t)\rightarrow 0$ as $t\rightarrow\infty$ then there exists $T>0$ such that for all $t\geq T$ we have that $\beta(r,t)\leq\varepsilon$. Let $t_0\ge 0$, $x\in B_r$ and $u \in \U$. Then, for all $t\ge t_0+T$, 
\begin{align*}
\|\phi(t,t_{0},x,u)\|_{\mathcal{X}}&\leq  \beta(\|x\|_{\mathcal{X}},t-t_{0})+\rho(\|u\|_{\U})\\
&\leq  \beta(r,T)+\rho(\|u\|_{\U}) \le \varepsilon + \rho(\|u\|_{\U})
\end{align*}
and therefore C\ref{item:attract}) holds with $\nu=\rho$.

Conversely, suppose that $\Sigma$ satisfies C\ref{item:ubrs})--C\ref{item:attract}).

Let $r\ge 1$. 
%
By C\ref{item:attract}) with $\varepsilon=1$ there exists $T_{1}>0$ such that if $t_0\ge 0$, $\|x\|_{\mathcal{X}}\leq r$ and $u\in \U$, then $\|\phi(t,t_{0},x,u)\|_{\mathcal{X}}\leq 1+\nu\left(\|u\|_{\U}\right)$ for all $t\geq t_{0}+T_{1}$. If, in addition, $\|u\|_{\U}\leq r$, then $\|\phi(t,t_{0},x,u)\|_{\mathcal{X}}\leq 1+\nu\left( r\right)$ for all $t\geq t_{0}+T_{1}$.

From C\ref{item:ubrs}), there exists a $C>0$ such that $\|\phi(t,t_{0},x,u)\|_{\mathcal{X}}\leq C$ if $t\in[t_{0},t_{0}+T_{1}]$, $\|x\|_{\mathcal{X}}\leq r$ and $\|u\|_{\U}\leq r$. Therefore, $\|\phi(t,t_{0},x,u)\|_{\mathcal{X}}\leq \max\{C,1+\nu\left( r\right)\}\leq 1+C+\nu\left( r\right)$ for all $t\ge t_0\ge 0$, $x\in \X$ such that $\|x\|_{\X} \le r$ and all $u\in \U$ such that $\|u\|_{\U}\le r$. 

Define for $r\ge 0$,
$$\varphi(r):=\sup\left\lbrace\|\phi(t,t_{0},x,u)\|_{\mathcal{X}}:\;0\le t_{0}\le t,\;\|x\|_{\mathcal{X}}\leq r,\;\|u\|_{\U}\leq r\right\rbrace.$$
Note that $\varphi$ is clearly nondecreasing and, by the previous analysis, $\varphi(r)$ is finite for every $r\geq 1$. Then, $\varphi(r)$ is finite for every $r\ge 0$.
By C\ref{item:smallx0u}) and C\ref{item:attract}), it straightforwardly follows that $\varphi(r)\rightarrow 0$ as $r\rightarrow 0^+$. Then, there exists $\hat\varphi \in \Ki$ such that $\varphi \le \hat\varphi$. Therefore  $\|\phi(t,t_{0},x,u)\|_{\mathcal{X}}\leq\varphi(\max\{\|x\|_{\mathcal{X}},\|u\|_{\U}\})\leq \hat \varphi(\|x\|_{\mathcal{X}})+ \hat \varphi(\|u\|_{\U})$ for all $x\in\mathcal{X}$, $u\in\mathcal{U}$ and $t\ge t_{0}\ge 0$.
It follows that for all $t\geq t_0 \ge 0$, all $x\in \X$ such that $\|x\|_{\X}\le r$ and all $u\in\mathcal{U}$ we have $\|\phi(t,t_{0},x,u)\|_{\mathcal{X}}\leq  \hat{\varphi}(r) + \hat{\varphi}(\|u\|_{\U})$. Then, for all $t\geq t_0 \ge 0$, all $x\in \X$ such that $\|x\|_{\X}\le r$ and all $u\in\mathcal{U}$,
\begin{align} \label{eq:bound}
\|\phi(t,t_{0},x,u)\|_{\mathcal{X}}\leq  \hat{\varphi}(r) + \hat{\varphi}(\|u\|_{\U}).
\end{align}
From (\ref{eq:bound}) and C3), by proceeding as in the proof of Lemma 8 in \cite{Mironchenko2018} or as in that of Lemma 2.7 in \cite{sonwan_scl95},
it follows that there exists a function $\beta\in \KL$ for which the estimate
\begin{equation*}
\|\phi(t,t_{0},x,u)\|_{\mathcal{X}}\leq\beta(\|x\|_{\mathcal{X}},t-t_{0})+\rho(\|u\|_{\U})
\end{equation*}
holds with $\rho := \max\{\nu,\hat\varphi\}$. Hence the system $\Sigma$ is ISS.

\section{Proof of Lemma \ref{lem:o-guas+UGR:UGS}}
\label{app:proof-lem:o-guas+UGR:UGS}

Let $\alpha$, $\rho$ and $c$ as in the definition of UGB. For $r\geq 0$, define
\begin{equation*}
\tilde{\alpha}(r)=\sup\{\|\phi(t,t_{0},x,u)\|_{\mathcal{X}}: 0\le t_0\le t,\;\|x\|_{\mathcal{X}}\leq r\;\makebox{and}\;\|u\|_{\U}\leq r\}
\end{equation*}
The function $\tilde{\alpha}$ is non-decreasing and finite by the UGB property. 

Next, we prove that $\displaystyle\lim_{r\rightarrow 0^{+}}\tilde{\alpha}(r)=0$. Define $r^{*}=\rho(1)+\alpha(1)+c$ and let $\beta\in\mathcal{KL}$ be the function characterizing 0-GUAS. For a given $\varepsilon >0$, take  $\delta_{1}\in(0,1)$ so that $\delta_{1}\leq\beta(\delta_{1},0)<\frac{\varepsilon}{2}$ and $T>0$ such that $\beta(\delta_{1},T)<\frac{\delta_{1}}{2}$. Define $\eta=\frac{\delta_{1}}{2}$ and let $\gamma=\gamma(r^{*},\eta,T)$ be the constant coming from Lemma \ref{lem:main}. Induction will be used to prove that for every $x_0\in\X$ such that $\|x_0\|_{\mathcal{X}}\leq \delta_{1}$, every $0\le t_{0}\le t$, and every $u\in\mathcal{U}$ with $\|u\|_{\U}\leq\gamma$, then $\|\phi(t,t_{0},x_0,u)\|_{X}<\varepsilon$ for all $t\geq t_{0}$. 

For $i\in\N$, define $t_i := t_0 + iT$ and $x_i := \phi(t_i,t_0,x,u)$.
We have $\|\phi(t,t_{0},x_0,u)\|_{\X}\le r^*$ for all $t\ge t_0$. Apply Lemma~\ref{lem:main} to obtain 
\begin{align*}
\|\phi(t,t_{0},x_0,u)\|_{\mathcal{X}}&\leq \beta(\|x_0\|_{\X},t-t_{0})+\eta
\leq  \beta(\delta_{1},0)+\eta\\
&\leq  \frac{\varepsilon}{2}+\eta=  \frac{\varepsilon}{2}+\frac{\delta_{1}}{2} \\
&<  \frac{\varepsilon}{2}+\frac{\varepsilon}{2}=\varepsilon
\end{align*}
for all $t \in [t_0,t_0+T]=[t_0,t_1]$. In addition, $\|x_1\|_{\X}=\|\phi(t_1,t_{0},x,u)\|_{\X}\leq \beta(\delta_{1},T)+\eta<\frac{\delta_{1}}{2}+\frac{\delta_{1}}{2}=\delta_{1}$. 

Next, suppose that $\|x_i\| \le \delta_1$. Using the fact that $\phi(t,t_i,x_i,u)=\phi(t,t_0,x_0,u)$ for all $t\ge t_i$, repeating the latter reasoning we obtain $\|\phi(t,t_0,x_0,u)\|_{\X}\le \varepsilon$ for all $t\in [t_i,t_{i+1}]$ and $\|x_{i+1}\| = \|\phi(t_{i+1},t_0,x,u)\|_{\X}\le \delta_1$. In consecuence, induction establishes that $\|\phi(t,t_{0},x_0,u)\|_{\mathcal{X}}<\varepsilon$ for all $t\geq t_{0} \ge 0$, $x_0\in\X$, and $u\in \U$, provided that $\|x_0\|_{\X} \leq \delta$ and $\|u\|_{\U}\leq\delta$, with $\delta=\min\{\delta_1,\gamma\}$.
By definition of $\tilde\alpha$, it follows that for every $\varepsilon > 0$ there exists $\delta > 0$ such that $\tilde{\alpha}(r)\leq\tilde{\alpha}(\delta)<\varepsilon$ for every $0<r<\delta$. This proves that $\displaystyle\lim_{r\rightarrow 0^{+}}\tilde{\alpha}(r)=0$. 

Since $\tilde{\alpha}$ is non-decreasing and $\displaystyle\lim_{r\rightarrow 0^{+}}\tilde{\alpha}(r)=0$, there exists $\hat{\alpha}\in\mathcal{K}_{\infty}$ such that $\hat{\alpha}(r)\geq\tilde{\alpha}(r)$ for all $r\geq 0$. 
Let $0\le t_0\le t$, $x\in\mathcal{X}$, and $u\in\U$. From the definition of $\tilde{\alpha}$ and the fact that $\hat{\alpha}(r)\geq\tilde{\alpha}(r)$, it follows that
\begin{eqnarray*}
\|\phi(t,t_{0},x,u)\|_{\mathcal{X}} &\leq & \widehat{\alpha}(\|x\|_{\mathcal{X}})+\widehat{\alpha}(\|u\|_{\U})
\end{eqnarray*}
Consequently, the system is $\|\cdot\|_{\U}$-UGS because~(\ref{eq:ugb}) holds with $c=0$ and $\alpha=\rho=\hat{\alpha}$.

\section{Proof of Lemma~\ref{lem:main}}
\label{app:proof-lem:main}

Suppose that system $\Sigma$ is $0$-GUAS and let $\beta\in \KL$ so that (\ref{eq:0-guas}) holds. Let $r>0$, $\eta>0$, $T>0$. Set $r^*=\beta(r,0)$ and note that $r\le r^*$. Let $\delta=\delta(r^*,\eta,T)$ be the positive constant given by Assumption~\ref{as:transition-map} with $r^*$ instead of $r$ and $\eta$ instead of $\varepsilon$. 

Suppose that $\|\phi(t,t_{0},x,u)\|_{\mathcal{X}}\leq r$ for all $t\in[t_{0},t_{0}+T]$ and that $\|u\|_{\U}\leq \delta$. Then, $\|\phi(t,t_{0},x,\zi)\|_{\X}\le \beta(r,0)\le r^*$ for all $t\in [t_{0},t_{0}+T]$. The definition of $\delta$ and Assumption~\ref{as:transition-map} imply that
\begin{equation*}
\|\phi(t,t_{0},x,u)-\phi(t,t_{0},x,\zi)\|_{\mathcal{X}}\leq \eta \quad\forall t\in[t_{0},t_{0}+T].
\end{equation*}
Thus
\begin{eqnarray*}
\|\phi(t,t_{0},x,u)\|_{\mathcal{X}} &\leq & \|\phi(t,t_{0},x,\zi)\|_{\mathcal{X}}+\|\phi(t,t_{0},x,u)-\phi(t,t_{0},x,\zi)\|_{\mathcal{X}}\\
&\leq & \beta(\|x\|_{\mathcal{X}},t-t_{0})+\eta.
\end{eqnarray*}
and the proof concludes taking $\gamma=\delta$.

\section{Proof of Lemma \ref{lem:teosigrh}} \label{app:p-ds}
The proof of Lemma \ref{lem:teosigrh} employs the following version of Gronwall Lemma.
\begin{lema} \label{lem:gronwall}
Let $\psi:[t_{0}-\tau,t]\rightarrow \R$ be continuous and nonnegative and let $K,L\ge 0$ be such that 
\begin{align*}
\psi(\ell)\leq K+L\int_{t_{0}}^{\ell}\|\psi_{s}\|ds \quad\forall \ell\in[t_{0},t],
\end{align*}
where $\psi_{s}\in \mathcal{C}([-\tau,0],\R)$ is the function defined by $\psi_s(r)=\psi(r+s)$ for all $r\in [-\tau,0]$ and $\|\cdot\|$ is the supremum norm. Then,
\begin{align*}
\|\psi_\ell\|\leq \left(K+\|\psi_{t_{0}}\|\right) e^{L(\ell-t_{0})} \quad\forall \ell\in[t_{0},t].
\end{align*}
\end{lema}
\begin{proof}
  Define for $\ell\in [t_0,t]$, $\varphi(\ell)=\|\psi_{\ell}\|$ and $\Phi(\ell)=  K+L\int_{t_{0}}^{\ell}\|\psi_{s}\|ds$. Note that $\varphi$ is nonnegative and continuous and that $\Phi$ is nondecreasing. For every $\ell\in [t_0,t]$ and any $s\in [-\tau,0]$ we have that $\psi_{\ell}(s)=\psi(s+\ell)\le \Phi(s+\ell)\le \Phi(\ell)$ when $s+\ell\ge t_0$ and that $\psi_{\ell}(s)=\psi(s+\ell)\le \|\psi_{t_0}\|$ when $s+\ell<t_0$. In consequence, $\varphi(\ell)\le \|\psi_{t_0}\|+\Phi(\ell)$ for all $\ell\in [t_0,t]$ and hence
  \begin{align*}
    \varphi(\ell)\leq  \|\psi_{t_{0}}\|+ K+ L\int_{t_{0}}^{\ell}\varphi(s)ds.
  \end{align*}
  Applying the standard Gronwall inequality yields
  \begin{align*}
    \varphi(\ell)=\|\psi_\ell\|\leq \left( K+\|\psi_{t_{0}}\| \right) e^{L(\ell-t_{0})} \quad \forall \ell\in [t_0,t],
  \end{align*}
  which establishes the result.
\end{proof}

The following lemma employs this version of Gronwall's inequality to give a bound on the difference between specific solutions.
\begin{lema}
  \label{lem6}
  Suppose that $f$ in (\ref{eq:sistr}) satisfies Assumption~\ref{as:f-delay} and let $\gamma$ be given by (R1). Then, for every $r>0$ and $\eta>0$ there exist $L=L(r)$ and $k=k(r,\eta)$ such that if $x(\cdot)$ and $z(\cdot)$ are the maximally defined solutions of (\ref{eq:sistr}) corresponding to initial time $t_0\ge 0$, initial state $\psi_0 \in \mathcal{C}$ and, respectively, inputs $u\in \U$ and $\zi\in\U$, and if for some time $t^*>t_0$ it happens that $\|x_t\|\le r$ and $\|z_t\|\le r$ for all $t\in [t_0,t^*]$, then it also happens that
\begin{align} \label{eq:lema6}
\|x_t-z_t\|\leq\left[\eta(t-t_{0})+k\int^{t}_{t_{0}}\gamma(|u(s)|)\:ds \right]e^{L(t-t_{0})} \quad t\in [t_0,t^*]. 
\end{align}
\end{lema}
\begin{proof}
  For every $s\ge 0$ define $\B_{s}^{\C}=\{\psi\in\mathcal{C}:\|\psi\|\leq s\}$ and $\B_{s}^{m}=\{\xi\in\R^m:|\xi|\leq s\}$.

  The following claim is analogous to that in the proof of \cite[Lemma~3]{haiman_tac18}.
  \begin{clm}
    \label{clm:keta}
    For every $r>0$ and $\eta >0$, there exists $k=k(r,\eta)>0$ such that for all $t\geq 0$, $\psi\in\B_{r}^{\C}$ and $\mu\in\R^{m}$
    \begin{align*}
      |f(t,\psi,\mu)-f(t,\psi,0)|\leq\eta + k \gamma(|\mu|).
    \end{align*}
  \end{clm} 
  \begin{proof}[Proof of the claim]
    Let $r>0$ and $\eta >0$ and take $\delta\in(0,1)$ from (R2) in Assumption~\ref{as:f-delay}, such that for all $t\geq 0$ and $(\psi,\mu)\in\B_{r}^{\C}\times\B_{\delta}^{m}$ then
    \begin{align*}
      |f(t,\psi,\mu)-f(t,\psi,0)|<\eta.
    \end{align*}
    If $\psi\in\B_{r}^{\C}$ and $|\mu|\geq\delta$, from (R1) in Assumption~\ref{as:f-delay}, it follows that
    \begin{align*}
      |f(t,\psi,\mu)-f(t,\psi,0)|&\leq  |f(t,\psi,\mu)|+|f(t,\psi,0)|\\
      &\leq  N(\|\psi\|)+N(\|\psi\|)\gamma(|\mu|)+N(\|\psi\|)\\
      &= 2N(\|\psi\|)+N(\|\psi\|)\gamma(|\mu|)\\
      &\leq  N(r) [2+\gamma(|\mu|)]\\
      &= N(r) \left[\frac{2}{\gamma(|\mu|)}+1\right]\gamma(|\mu|) \\
      &\leq N(r) \left[\frac{2}{\gamma(\delta)}+1\right]\gamma(|\mu|).
    \end{align*}
    By taking $k=N(r)\left[\frac{2}{\gamma(\delta)}+1\right]$ we then have that for all $t\ge 0$, $\psi \in B_r^\C$ and $\mu \in \R^m$,
    \begin{align*}
      |f(t,\psi,\mu)-f(t,\psi,0)|\leq \eta+ k \gamma(|\mu|)
    \end{align*}
    and the claim follows.
  \end{proof}

  Let $r>0$, $\eta > 0$, and let $L=L(r)$ be given by (R3) in Assumption~\ref{as:f-delay}.
  Let $k=k(r,\eta)$ be given by Claim~\ref{clm:keta} and let $t_0$, $t^*$, $\psi$, $u$, $x(\cdot)$ and $z(\cdot)$ be as in the statement of Lemma \ref{lem6}. 
  Let $t_0\le t\le t^*$. Since for $\ell\in [t_0,t]$,
  \begin{align*}
    x(\ell)=x(t_0)+\int_{t_0}^{\ell} f(s,x_s,u(s))\:ds
  \end{align*}
  and 
  \begin{align*}
    z(\ell)=x(t_0)+\int_{t_0}^{\ell} f(s,z_s,0)\: ds
  \end{align*}
  it follows that for all $\ell \in [t_0,t]$,
  \begin{align*}
    |x(\ell)-z(\ell)|&\leq \int_{t_{0}}^{\ell}|f(s,x_{s},u(s))-f(s,z_{s},0)|ds\\
    &\leq \int_{t_{0}}^{\ell}|f(s,x_{s},u(s))-f(s,x_{s},0)|ds
    +\int_{t_{0}}^{\ell}|f(s,x_{s},0)-f(s,z_{s},0)|ds\\
    &\leq  \int_{t_{0}}^{\ell}\eta + k \gamma(|u(s)|)ds+L\int_{t_{0}}^{\ell}\| x_{s}-z_{s}\| ds\\
    &\leq  \eta(\ell -t_{0})+k\int_{t_{0}}^{\ell}\gamma(|u(s)|)ds+L\int_{t_{0}}^{\ell}\| x_{s}-z_{s}\| ds
  \end{align*} 
  Let $\varphi(\ell)=|x(\ell)-z(\ell)|$ for $\ell \in [t_0-\tau,t]$. Then, for all $\ell \in [t_0,t]$,
  \begin{equation*}
    \varphi(\ell)\leq \eta(t -t_{0})+k\int_{t_{0}}^{t} \gamma(|u(s)|) ds + \int_{t_{0}}^{\ell} L \|\varphi_s\| ds.
  \end{equation*}
  Applying Lemma~\ref{lem:gronwall} to $\varphi$ with $K=\eta(t -t_{0})+k\int_{t_{0}}^{t} \gamma(|u(s)|)ds$, and taking into account that $\varphi_{t_0}\equiv 0$ because $x_{t_0} = z_{t_0} = \psi_0$, then (\ref{eq:lema6}) follows, concluding the proof of Lemma~\ref{lem6}.
\end{proof}
\begin{proof}[Proof of Lemma \ref{lem:teosigrh}]
  Given $r>0$, $\varepsilon>0$ and $T>0$, let $L=L(r)$ be given by Lemma~\ref{lem6}. Pick $\eta>0$ sufficiently small such that $\eta T e^{LT}<\varepsilon/2$ and let $k=k(r,\eta)$ be given by Lemma~\ref{lem6}. Pick $\delta>0$ such that $k \delta e^{LT}<\varepsilon/2$. 

  Suppose that for $t_0\le t^* \le t_0+T$, $\psi\in \C$ and $u\in \U$ such that $\|u\|_{\gamma}<\delta$, the maximal solutions $x(\cdot)$ and $z(\cdot)$ of (\ref{eq:sistr}) corresponding to $t_0$, $\psi$ and inputs $u$ and $\zi$, respectively, are defined on $[t_0-\tau,t^*]$ and satisfy $\|x_t\|\le r$ and $\|z_t\|\le r$ for all $t\in [t_0,t^*]$. From Lemma~\ref{lem6}, it follows that for all $t\in [t_0,t^*]$
  \begin{align*}
    \|x_t-z_t\|&\leq\left[\eta(t-t_{0})+k\int^{t}_{t_{0}}\gamma(|u(s)|)\:ds \right]e^{L(t-t_{0})} \\
    &\le \eta T e^{LT}+k \|u\|_{\gamma} e^{LT} 
    \le \frac{\varepsilon}{2}+ \frac{\varepsilon}{2}=\varepsilon.
  \end{align*}
  The proof finishes by noting that if $\phi$ is the transition map of $\Sigma^R$ then $\phi(t,t_0,\psi,u)=x_t$ and $\phi(t,t_0,\psi,\zi)=z_t$ for all $t\in [t_0,t^*]$.
\end{proof}

{\color{black} \section{Proof of Lemma \ref{lem:BIC}} \label{app:BIC}
Let $t_0\ge 0$, $x_0\in \X$, $u \in \U$ and $x(\cdot)$ be the corresponding solution of (\ref{eq:sistsl}). Let $[t_0,t_{(t_0,x_0,u)})$ be the maximal interval of definition of $x(\cdot)$. Suppose that $t_{(t_0,x_0,u)}<\infty$ and that $\|x(t)\|_{\X}\le M$ for all $t\in [t_0,t_{(t_0,x_0,u)})$ for some $M\ge 0$. Since $u(\cdot)$ is piecewise continuous, we have that it is bounded on $I=[t_0, t_{(t_0,x_0,u)}+1]
$, so that $r_u := \sup_{t\in I} \|u(t)\|_{\Ui} < \infty$. We claim that $F(t,\xi)=f(t,\xi,u(t))$ is Lipschitz in $\xi$ on bounded sets, uniformly in $t \in I$, with a Lipschitz constant that may depend on $r_u$ and that $F(t,0)$ is bounded on $I$. To see this, let $r>0$ and select $L=L(\max\{r,r_u\})$ from (SL2), so that the function $F$ satisfies
\begin{align*}
  \| F(t,\xi) - F(t,\omega) \|_{\X} = \| f(t,\xi,u(t)) - f(t,\omega,u(t)) \|_{\X} \le L \| \xi - \omega \|_{\X}
\end{align*}
for all $t\in I$, whenever $\|\xi\|_\X \le r$, $\|\omega\|_\X \le r$. 

From (SL1) and the fact that $u(\cdot)$ is piecewise continuous, we have that $t\mapsto F(t,0)$ is piecewise continuous on $I$ and therefore it is bounded, This proves the claim.

Since $x(\cdot)$ is a maximally defined weak solution of (\ref{eq:sistsl}) with $F(t,x(t))$ instead of $f(t,x(t),u(t))$, $t_{(t_0,x_0,u)}<\infty$, $F(t,\xi)$ satisfies a Lipschitz condition and $F(t,0)$ is bounded on $I$, then a slight variation of \cite[Thm.4.3.4]{cazenave_1998} implies that $x(\cdot)$ is unbounded on $[t_0, t_{(t_0,x_0,u)})$. This is a contradiction showing that $t_{(t_0,x_0,u)}<\infty$ is not possible when $x(\cdot)$ is bounded on $[t_0, t_{(t_0,x_0,u)})$. Then, $t_{(t_0,x_0,u)}=\infty$ and the BIC property follows. }

\section{Proof of Lemma \ref{lem:lema-ass2-ssl}} \label{app:p-sls}
For proving Lemma \ref{lem:lema-ass2-ssl} we use the fact that since $T(\cdot)$ is a strongly continuous semigroup, there exist $M>0$ and $w\ge 0$ such that the operator norm $\|T(t)\|\le M e^{wt}$ for all $t\ge 0$. We also need the following result, which is analogous to Lemma \ref{lem6}.
\begin{lema}\label{lem:SL3}
Suppose that the function $f$ in (\ref{eq:sistsl}) satisfies Assumption~\ref{ass:f-sml} and let $\gamma$ be given by (SL3). 
Then, for every $r>0$ and $\eta>0$ there exist $L=L(r)$ and $k= k(r,\eta)$ such that if $x(\cdot)$ and $z(\cdot)$ are the maximally defined solutions of (\ref{eq:sistsl}) corresponding to initial time $t_0\ge 0$, initial state $x_0 \in \X$ and the inputs $u\in \U$ and $\zi\in\U$, respectively, and if for some time $t^*>t_0$, $\|x(t)\|_{\X}\le r$ and $\|z(t)\|_{\X}\le r$ for all $t\in [t_0,t^*]$, then we have that
\begin{align} \label{eq:SL3}
\|x(t)-z(t)\|_{\X}\leq\left[\eta(t-t_{0})+ k\int^{t}_{t_{0}}\gamma(|u(s)|)\:ds \right]Me^{LMe^{w(t-t_0)}+w(t-t_{0})} \quad t\in [t_0,t^*]. 
\end{align}
\end{lema}
\begin{proof}
  The following Claim, which is analogous to that in the proof of Lemma \ref{lem6}, can be proved in the same way, but using (SL3)--(SL4) instead of (R1)--(R2).
  \begin{clm}
    \label{clm:semilin}
    For every $r>0$ and $\eta >0$, there exists $k=k(r,\eta)>0$ such that for all $t\geq 0$, $x\in\B_{r}^{\X}$ and $\mu\in\Ui$
    \begin{align*}
      \|f(t,x,\mu)-f(t,x,0)\|_{\X}\leq\eta + k \gamma(\|\mu\|_{\Ui}).
    \end{align*}
  \end{clm}
  Let $r>0$ and let $L=L(r)$ be given by (SL2) in Assumption~\ref{ass:f-sml}. Let $\eta>0$. Let $k=k(r,\eta)$ be given by Claim~\ref{clm:semilin} and let $t_0$, $t^*$, $\psi$, $u$, $x(\cdot)$ and $z(\cdot)$ be as in the statement of Lemma~\ref{lem:SL3}. 
  For $t_0\le t \le t^*$ and $\tau \in [t_0,t]$, we have that
  \begin{align*}
    x(\tau)=T(\tau-t_0)x_0+\int_{t_0}^{\tau} T(\tau-s)f(s,x(s),u(s))\:ds
  \end{align*}
  and 
  \begin{align*}
    z(\tau)=T(\tau-t_0)x_0+\int_{t_0}^{\tau} T(\tau-s)f(s,z(s),0)\:ds.
  \end{align*}
  Then, using the operator bound $\|T(h)\|\le Me^{wh}$ for all $h\ge 0$ and Claim~\ref{clm:semilin}, it follows that for all $t_0\le \tau \le t \le t^*$
  \begin{align*}
    \lefteqn{\|x(\tau)-z(\tau)\|_{\mathcal{X}} \leq \int_{t_{0}}^{\tau}\|T(\tau -s)\|\|f(s,x(s),u(s))-f(s,z(s),0)\|_{\mathcal{X}}ds}
    \hspace{5mm} &\\
    &\leq \int_{t_{0}}^{\tau}Me^{w(\tau -s)}\|f(s,x(s),u(s))-f(s,x(s),0)\|_{\mathcal{X}}ds\\
    &\quad +\int_{t_{0}}^{\tau}Me^{w(\tau -s)}\|f(s,x(s),0)-f(s,z(s),0)\|_{\mathcal{X}}ds\\
    &\leq  \int_{t_{0}}^{\tau}Me^{w(\tau -s)} \left[\eta + k \gamma(\|u(s)\|_{\mathsf{U}})\right] ds+L\int_{t_{0}}^{\tau}Me^{w(\tau -s)}\|x(s)-z(s)\|_{\mathcal{X}}ds\\
    &\leq  \int_{t_{0}}^{\tau}Me^{w(t-t_0)} \left[\eta + k \gamma(\|u(s)\|_{\mathsf{U}})\right] ds+L\int_{t_{0}}^{\tau}Me^{w(t-t_0)}\|x(s)-z(s)\|_{\mathcal{X}}ds\\
    &\leq Me^{w(t-t_0)} \left[\eta(t -t_{0})+k\int_{t_{0}}^{t}\gamma(\|u(s)\|_{\Ui})ds \right]
    +LMe^{w(t-t_0)}\int_{t_{0}}^{\tau}\| x(s)-z(s)\|_{\mathcal{X}}ds
  \end{align*}
  By applying Gronwall Lemma on the interval $[t_0,t]$, (\ref{eq:SL3}) follows.
\end{proof}
\begin{proof}[{\bf Proof of Lemma \ref{lem:lema-ass2-ssl}}]
  The proof is analogous to that of Lemma \ref{lem:teosigrh}, but using Lemma \ref{lem:SL3} instead of Lemma \ref{lem6}.
\end{proof}

\bibliographystyle{siamplain}
\bibliography{/home/hhaimo/latex/strings,biblio}
\end{document}